\documentclass[11pt,a4paper]{amsart}
\usepackage{graphicx}
\usepackage{amsmath,amsfonts}
\usepackage{amsthm,amssymb,latexsym}
\usepackage[active]{srcltx}
\usepackage[usenames]{color}

\newcommand{\bfs}{\boldsymbol}

\newtheorem{theorem}{Theorem}[section]
\newtheorem{corollary}[theorem]{Corollary}
\newtheorem{lemma}[theorem]{Lemma}

\newtheorem*{fact*}{Fact}
\newtheorem{proposition}[theorem]{Proposition}
\theoremstyle{definition}

\newtheorem{remark}[theorem]{Remark}
\numberwithin{equation}{section}

\newcommand{\N}{\mathbb N}
\newcommand{\Z}{\mathbb Z}
\newcommand{\G}{\mathbb G}

\newcommand{\A}{\mathbb A}
\newcommand{\F}{\mathbb F}
\newcommand{\K}{\mathbb K}
\newcommand{\R}{\mathbb R}
\newcommand{\Pp}{\mathbb P}

\newcommand{\fq}{\F_{\hskip-0.7mm q}}

\newcommand{\fqtwo}{\F_{\hskip-0.7mm q^{2}}}
\newcommand{\fqi}{\F_{\hskip-0.7mm q^i}}
\newcommand{\cfq}{\overline{\F}_{\hskip-0.7mm q}}
\def\ifm#1#2{\relax \ifmmode#1\else#2\fi}

\newcommand{\klk}    {\ifm {,\ldots,} {$,\ldots,$}}

\newcommand{\plp}    {\ifm {+\cdots+} {$+\ldots+$}}

\newcommand{\wt}    {\ifm {{\sf wt}} {{$\sf wt$}}}

\begin{document}

\title[Symmetric Complete Intersections]{Number of rational points
of symmetric complete intersections over a finite field and
applications}
\author[G. Matera et al.]{
%
%\author[G. Matera]{
Guillermo Matera${}^{1,2}$,
%
%\author[M. P\'erez]{
Mariana P\'erez${}^1$,
%
%\author[M. Privitelli]{
and Melina Privitelli${}^3$}

\address{${}^{1}$Instituto del Desarrollo Humano,
Universidad Nacional de Gene\-ral Sarmiento, J.M. Guti\'errez 1150
(B1613GSX) Los Polvorines, Buenos Aires, Argentina}
\email{\{gmatera,\,vperez\}@ungs.edu.ar}
\address{${}^{2}$ National Council of Science and Technology (CONICET),
Ar\-gentina}%\email{mprivitelli@conicet.gov.ar}
\address{${}^{1}$Instituto de Ciencias,
Universidad Nacional de Gene\-ral Sarmiento, J.M. Guti\'errez 1150
(B1613GSX) Los Polvorines, Buenos Aires, Argentina}
\email{mprivite@ungs.edu.ar}

\thanks{The authors were partially supported by the grants
PIP CONICET 11220130100598, PIO CONICET-UNGS 14420140100027 and UNGS
30/3084.}


\date{\today}%

\begin{abstract}
We study the set of common $\fq$--rational zeros of systems of
multivariate symmetric polynomials with coefficients in a finite
field $\fq$. We establish certain properties on these polynomials
which imply that the corresponding set of zeros over the algebraic
closure of $\fq$ is a complete intersection with ``good'' behavior
at infinity, whose singular locus has a codimension at least two or
three. These results are used to estimate the number of
$\fq$--rational points of the corresponding complete intersections.
Finally, we illustrate the interest of these estimates through their
application to certain classical combinatorial problems over finite
fields.
\end{abstract}

\maketitle
%
%----------------------------------------------------------------------
%----------------------------------------------------------------------
%----------------------------------------------------------------------
%----------------------------------------------------------------------
%----------------------------------------------------------------------
%----------------------------------------------------------------------
%----------------------------------------------------------------------
%----------------------------------------------------------------------
%
\section{Introduction}
Several problems of coding theory, cryptography or combinatorics
require the study of the set of rational points of varieties defined
over a finite field $\fq$ on which the symmetric group of
permutations of the coordinates acts. In coding theory, deep holes
in the standard Reed--Solomon code over $\fq$ can be expressed by
the set of zeros with coefficients in $\fq$ of certain symmetric
polynomial (see, e.g., \cite{ChMu07} or \cite{CaMaPr12}). Further,
the study of the set of $\fq$--rational zeros of a certain class of
symmetric polynomials is fundamental for the decoding algorithm for
the standard Reed--Solomon code over $\fq$ of \cite{Sidelnikov94}.
In cryptography, the characterization of monomials defining an
almost perfect nonlinear polynomial or a differentially uniform
mapping can be reduced to estimate the number of $\fq$--rational
zeros of certain symmetric polynomials (see, e.g., \cite{Rodier09}
or \cite{AuRo09}). In \cite{GoMa15}, an optimal representation for
the set of $\fq$--rational points of the trace--zero variety of an
elliptic curve defined over $\fq$ has been obtained by means of
symmetric polynomials. Finally, it is also worth mentioning that
applications in combinatorics over finite fields, as the
determination of the average cardinality of the value set and the
distribution of factorization patterns of families of univariate
polynomials with coefficients in $\fq$, has also been expressed in
terms of the number of common $\fq$--rational zeros of symmetric
polynomials defined over $\fq$ (see \cite{CeMaPePr14} and
\cite{CeMaPe15}).

In \cite{CaMaPr12}, \cite{CeMaPePr14} and \cite{CeMaPe15} we have
developed a methodology to deal with some of the problems mentioned
above. This methodology relies on the study of the geometry of the
set of common zeros of the symmetric polynomials under consideration
over the algebraic closure $\cfq$ of $\fq$. By means of such a study
we have been able to prove that in all the cases under consideration
the set of common zeros in $\cfq$ of the corresponding symmetric
polynomials is a complete intersection, whose singular locus has a
``controlled'' dimension. This has allowed us to apply certain
explicit estimates on the number of $\fq$--rational zeros of
projective complete intersections defined over $\fq$ (see, e.g.,
\cite{GhLa02a}, \cite{CaMa07}, \cite{CaMaPr15} or \cite{MaPePr16})
to obtain a conclusion for the problem under consideration.

The analysis of \cite{CaMaPr12}, \cite{CeMaPePr14} and
\cite{CeMaPe15} has several points in common, which may be put on a
common basis that might be useful for other problems. For this
reason, in this paper we present a unified and generalized framework
where a study of the geometry of complete intersections defined by
symmetric polynomials with coefficients in $\fq$ can be carried out
along the lines of the papers above. More precisely, let $Y_1\klk
Y_s$ be indeterminates over $\fq$ and let $S_1\klk
S_m\in\fq[Y_1,\ldots,Y_s]$. Consider the weight
$\wt:\fq[Y_1,\ldots,Y_s]\to\N$ defined by setting $\wt(Y_j):=j$ for
$1\leq j \leq s$ and denote by $S_1^{\wt}\klk S_m^{\wt}$ the
components of highest weight of $S_1\klk S_m$. Let
$(\partial\bfs{S}/
\partial \bfs{Y})$ and $(\partial\bfs{S}^{\wt}/
\partial \bfs{Y})$ the Jacobian matrices of
$S_1\klk S_m$ and $S_1^{\wt}\klk S_m^{\wt}$ with respect to $Y_1\klk
Y_s$ respectively. Suppose that
\begin{itemize}
\item[(1)]  $(\partial\bfs{S}/
\partial \bfs{Y})$ has full rank on every point of the affine variety
$V(S_1\klk S_m)\subset\A^s$,
\item[(2)]  $(\partial\bfs{S}^{\wt}/
\partial \bfs{Y})$ has full rank on every point of the affine variety $V(S_1^{\wt}\klk
S_m^{\wt})\subset\A^s$.
\end{itemize}
Finally, if $X_1,\ldots,X_r$ are new indeterminates over $\fq$ and
$\Pi_1,\ldots,\Pi_s$ are the first $s$ elementary symmetric
polynomials of $\fq[X_1,\ldots,X_r]$, we consider the polynomials
$R_1,\ldots,R_m\in\fq[X_1,\ldots,X_r]$ defined as
$$R_i:=S_i(\Pi_1,\ldots,\Pi_s)\quad (1\leq i \leq m).$$
We shall prove that the affine variety $V:=V(R_1\klk
R_m)\subset\A^r$ is a complete intersection whose geometry can be
studied with similar arguments as those in the papers cited above.
Further, we shall show how estimates on the number of
$\fq$--rational points of projective complete intersections can be
applied in this unified framework to estimate the number of
$\fq$--rational points of $V$. As a consequence, we obtain the
following result.
\begin{theorem}\label{theorem: main result - intro}
Let assumptions and notations be as above. Denote
$\delta:=\prod_{i=1}^s\deg R_i$ and $D:=\sum_{i=1}^s(\deg R_i-1)$.
If $|V(\fq)|$ is the cardinality of the set $V(\fq)$ of
$\fq$--rational points of $V$, then the following estimate holds:
$$\big||V(\fq)|-q^{r-m}\big|\le  14 D^3 \delta^2(q+1)q^{r-m-2}.$$
\end{theorem}

To illustrate the applications of Theorem \ref{theorem: main result
- intro}, we shall consider two classical combinatorial problems
over finite fields. The first one is concerned with the distribution
of factorization patterns on a linear family of monic polynomials of
given degree of $\fq[T]$. If the linear family under consideration
consists of the monic polynomials of degree $n$ for which the first
$s<n-2$ coefficients are fixed, then we have a function--field
analogous to the classical conjecture on the number of primes in
short intervals, which has been the subject of several articles
(see, e.g., \cite{Pollack13}, \cite{BaBaRo15} or \cite[\S
3.5]{MuPa13}). The study of general linear families of polynomial
goes back at least to \cite{Cohen72}. Here, following the approach
of \cite{CeMaPe15}, we obtain an explicit estimate on the number of
elements on a linear family of monic polynomials of fixed degree of
$\fq[T]$ having a given factorization pattern, for which we rely on
Theorem \ref{theorem: main result - intro}.

The second problem we consider is that of estimating the average
cardinality of the value sets of families of polynomials of $\fq[T]$
with certain consecutive coefficients prescribed. The study of the
cardinality of value sets of univariate polynomials is the subject
of the seminal paper \cite{BiSD59}. In \cite{Uchiyama55b} and
\cite{Cohen73} the authors determine the average cardinality of the
value set of all monic polynomials in $\fq[T]$ of given degree,
while \cite{Uchiyama55b} and \cite{Cohen72} are concerned with the
asymptotic behavior of the average cardinality of the value set of
all polynomials of $\fq[T]$ of given degree were certain consecutive
coefficients are fixed. We shall follow the approach of
\cite{CeMaPePr14}, where the question is expressed in terms of the
number of $\fq$--rational points of certain complete intersection
defined by symmetric polynomials.

The paper is organized as follows. In Section \ref{section:
notation, notations} we briefly recall the notions and notations of
algebraic geometry we use. Section \ref{section: geometry symm
complete inters} is devoted to present our unified framework and to
establish several results on the geometry of the affine complete
intersections of interest. In Section \ref{section: estimates
complete inters} we study the behavior of these complete
intersections ``at infinity'' and prove Theorem \ref{theorem: main
result - intro}. Finally, in Sections \ref{section: distribution of
fact patterns} and \ref{section: average card value sets} we apply
Theorem \ref{theorem: main result - intro} to determine the
distribution of factorization patterns and the average cardinality
of value sets of families of univariate polynomials.
%
%----------------------------------------------------------------------
%----------------------------------------------------------------------
%----------------------------------------------------------------------
%----------------------------------------------------------------------
%----------------------------------------------------------------------
%----------------------------------------------------------------------
%----------------------------------------------------------------------
%----------------------------------------------------------------------
%
\section{Notions, notations and preliminary results}
\label{section: notation, notations}
We use standard notions and notations of commutative algebra and
algebraic geometry as can be found in, e.g., \cite{Harris92},
\cite{Kunz85} or \cite{Shafarevich94}.

Let $\K$ be any of the fields $\fq$ or $\cfq$. We denote by $\A^n$
the $n$--dimensional affine space $\cfq{\!}^{n}$ and by $\Pp^n$ the
$n$--dimensional projective space over $\cfq{\!}^{n+1}$. %Both spaces
%are endowed with their respective Zariski topologies over $\K$, for
%which a closed set is the zero locus of a set of polynomials of
%$\K[X_1,\ldots, X_{n}]$, or of a set of homogeneous polynomials of
%$\K[X_0,\ldots, X_{n}]$.
%
By a {\em projective variety defined over} $\K$ (or a projective
$\K$--variety for short) we mean a subset $V\subset \Pp^n$ of common
zeros of homogeneous polynomials $F_1,\ldots, F_m \in\K[X_0 ,\ldots,
X_n]$. Correspondingly, an {\em affine variety of $\A^n$ defined
over} $\K$ (or an affine $\K$--variety) is the set of common zeros
in $\A^n$ of polynomials $F_1,\ldots, F_{m} \in
\K[X_1,\ldots, X_{n}]$. %We think a projective or affine
%$\K$--variety to be equipped with the induced Zariski topology.
We shall frequently denote by $V(F_1\klk F_m)$ or $\{F_1=0\klk
F_m=0\}$ the affine or projective $\K$--variety consisting of the
common zeros of the polynomials $F_1\klk F_m$.

In what follows, unless otherwise stated, all results referring to
varieties in general should be understood as valid for both
projective and affine varieties. A $\K$--variety $V$ is $\K$--{\em
irreducible} if it cannot be expressed as a finite union of proper
$\K$--subvarieties of $V$. Further, $V$ is {\em absolutely
irreducible} if it is $\cfq$--irreducible as a $\cfq$--variety. Any
$\K$--variety $V$ can be expressed as an irredundant union
$V=\mathcal{C}_1\cup \cdots\cup\mathcal{C}_s$ of irreducible
(absolutely irreducible) $\K$--varieties, unique up to reordering,
which are called the {\em irreducible} ({\em absolutely
irreducible}) $\K$--{\em components} of $V$.

%The set $V(\fq):=V\cap \fq^n$ is the set of {\sf $q$--rational
%points} of $V$. Studying the number of elements of $V(\fq)$ is a
%classical problem. The existence of $q$--rational points depends
%upon many circumstances concerning the geometry of the underlying
%variety.

For a $\K$--variety $V$ contained in $\Pp^n$ or $\A^n$, we denote by
$I(V)$ its {\em defining ideal}, namely the set of polynomials of
$\K[X_0,\ldots, X_n]$, or of $\K[X_1,\ldots, X_n]$, vanishing on
$V$. The {\em coordinate ring} $\K[V]$ of $V$ is defined as the
quotient ring $\K[X_0,\ldots,X_n]/I(V)$ or
$\K[X_1,\ldots,X_n]/I(V)$. The {\em dimension} $\dim V$ of $V$ is
the length $r$ of the longest chain $V_0\varsubsetneq V_1
\varsubsetneq\cdots \varsubsetneq V_r$ of nonempty irreducible
$\K$--varieties contained in $V$. We say that $V$ has {\em pure
dimension} $r$ if all the irreducible $\K$--components of $V$ are of
dimension $r$.
%
%A $\K$--variety in $\Pp^n$ or $\A^n$ of pure dimension $n-1$ is
%called a $\K$--{\em hypersurface}. A $\K$--hypersurface in $\Pp^n$
%(or $\A^n$) is the set of zeros of a single nonzero polynomial of
%$\K[X_0\klk X_n]$ (or of $\K[X_1\klk X_n]$).
%
%----------------------------------------------------------------
%----------------------------------------------------------------
%
%\paragraph{Degree}
%

The {\em degree} $\deg V$ of an irreducible $\K$-variety $V$ is the
maximum number of points lying in the intersection of $V$ with a
linear space $L$ of codimension $\dim V$, for which $V\cap L$ is a
finite set. More generally, following \cite{Heintz83} (see also
\cite{Fulton84}), if $V=\mathcal{C}_1\cup\cdots\cup \mathcal{C}_s$
is the decomposition of $V$ into irreducible $\K$--components, we
define the degree of $V$ as
$$\deg V:=\sum_{i=1}^s\deg \mathcal{C}_i.$$
%
%The degree of a $\K$--hypersurface $V$ is the degree of a polynomial
%of minimal degree defining $V$. Another property is that the degree
%of a dense open subset of a $\K$--variety $V$ is equal to the degree
%of $V$.
%
We shall use the following {\em B\'ezout inequality} (see
\cite{Heintz83}, \cite{Fulton84}, \cite{Vogel84}): if $V$ and $W$
are $\K$--varieties of the same ambient space, then
\begin{equation}\label{eq: Bezout}
\deg (V\cap W)\le \deg V \cdot \deg W.
\end{equation}

%Another result we shall use concerns the behavior of degree under
%linear mappings. Let $V\subset\Pp^m$ and $W\subset\Pp^n$ be
%$\K$--varieties and let $\phi:V\to W$ be a regular linear map. Then
%(see, e.g., \cite[Lemma 2.1]{CaMa07})
%%
%  \begin{equation}\label{eq:degree linear projection}
%    \deg \overline{\phi (V)} \leq \deg V,
%  \end{equation}
%%
%where $\overline{\phi (V)}$ is the Zariski closure of $\phi(V)$ in
%$\Pp^n$, $\deg \overline{\phi (V)}$ denotes the degree of
%$\overline{\phi (V)}$ as a $\K$--subvariety of $\Pp^n$ and $\deg V$
%denotes the degree of $V$ as a $\K$--subvariety of $\Pp^m$.
%
%----------------------------------------------------------------
%----------------------------------------------------------------
%
%\paragraph{Singular locus}

Let $V\subset\A^n$ be a $\K$--variety and $I(V)\subset
\K[X_1,\ldots, X_n]$ its defining ideal. Let $x$ be a point of $V$.
The {\em dimension} $\dim_xV$ {\em of} $V$ {\em at} $x$ is the
maximum of the dimensions of the irreducible $\K$--components of $V$
that contain $x$. If $I(V)=(F_1,\ldots, F_m)$, the {\em tangent
space} $T_xV$ {\em to $V$ at $x$} is the kernel of the Jacobian
matrix $(\partial F_i/\partial X_j)_{1\le i\le m,1\le j\le n}(x)$ of
the polynomials $F_1,\ldots, F_m$ with respect to $X_1,\ldots, X_n$
at $x$. %We have (see, e.g., \cite[page 94]{Shafarevich94})
%
%$$\dim T_xV\ge \dim_xV.$$
%
The point $x$ is {\em regular} if $\dim T_xV=\dim_xV$. Otherwise,
the point $x$ is called {\em singular}. The set of singular points
of $V$ is the {\em singular locus} $\mathrm{Sing}(V)$ of $V$; a
variety is called {\em nonsingular} if its singular locus is empty.
For a projective variety, the concepts of tangent space, regular and
singular point can be defined by considering an affine neighborhood
of the point under consideration.
%
%----------------------------------------------------------------
%----------------------------------------------------------------
%
%\paragraph{Mappings}

%Regular maps will be represented by solid arrows $\to$, while
%partial rational maps will be indicated with dashed arrows
%$\dashrightarrow$.
Let $V$ and $W$ be irreducible affine $\K$--varieties of the same
dimension and $f:V\to W$ a regular map for which
$\overline{f(V)}=W$, where $\overline{f(V)}$ is the closure of
$f(V)$ with respect to the Zariski topology of $W$. Such a map is
called {\em dominant}. Then $f$ induces a ring extension
$\K[W]\hookrightarrow \K[V]$ by composition with $f$. We say that
the dominant map $f$ is a {\em finite morphism} if this extension is
integral. %, i.e., each element $\eta\in\K[V]$ satisfies a monic
%equation with coefficients in $\K[W]$. A basic fact is that a
%dominant finite morphism is necessarily closed. Another fact
We observe that the preimage $f^{-1}(S)$ of an irreducible closed
subset $S\subset W$ under a dominant finite morphism is of pure
dimension $\dim S$ (see, e.g., \cite[\S 4.2,
Proposition]{Danilov94}).

Elements $F_1,\ldots, F_{n-r}$ in $\mathbb{K}[X_1,\ldots,X_n]$ or
$\mathbb{K}[X_0,\ldots,X_n]$ form a \emph{regular sequence} if $F_1$
is nonzero and no $F_i$ is a zero divisor in the quotient ring
$\mathbb{K}[X_1,\ldots,X_n]/ (F_1,\ldots,F_{i-1})$ or
$\mathbb{K}[X_0,\ldots,X_n]/ (F_1,\ldots,F_{i-1})$ for $2\leq i \leq
n-r$. In that case, the (affine or projective) $\mathbb{K}$--variety
$V:=V(F_1,\ldots,F_{n-r})$ is called a {\em set--theoretic complete
intersection}. Furthermore, $V$ is called an (ideal--theoretic) {\em
complete intersection} if its ideal $I(V)$ over $\K$ can be
generated by $n-r$ polynomials. If $V\subset\Pp^n$ is a complete
intersection of dimension $r$ defined over $\K$, and $F_1 \klk
F_{n-r}$ is a system of homogeneous generators of $I(V)$, the
degrees $d_1\klk d_{n-r}$ depend only on $V$ and not on the system
of generators. Arranging the $d_i$ in such a way that $d_1\geq d_2
\geq \cdots \geq d_{n-r}$, we call $\boldsymbol{d}:=(d_1\klk
d_{n-r})$ the {\em multidegree} of $V$. According to the {\em
B\'ezout theorem} (see, e.g., \cite[Theorem 18.3]{Harris92}), in
such a case we have $\deg V=d_1\cdots d_{n-r}$.

In what follows we shall deal with a particular class of complete
intersections, which we now define. A $\K$--variety $V$ is {\em
regular in codimension $m$} if its singular locus $\mathrm{Sing}(V)$
has codimension at least $m+1$ in $V$, i.e., $\dim V-\dim
\mathrm{Sing}(V)\ge m+1$. A complete intersection $V$ which is
regular in codimension 1 is called {\em normal} (actually, normality
is a general notion that agrees on complete intersections with the
one we define here). A fundamental result for projective normal
complete intersections is the Hartshorne connectedness theorem (see,
e.g., \cite[Theorem VI.4.2]{Kunz85}), which we now state. If
$V\subset\Pp^n$ is a complete intersection defined over $\K$ and
$W\subset V$ is any $\K$--subvariety of codimension at least 2, then
$V\setminus W$ is connected in the Zariski topology of $\Pp^n$ over
$\K$. Applying the Hartshorne connectedness theorem with
$W:=\mathrm{Sing}(V)$, one deduces the following result.
\begin{theorem}\label{theorem: normal complete int implies irred}
If $V\subset\Pp^n$ is a normal complete intersection, then $V$ is
absolutely irreducible.
\end{theorem}

%----------------------------------------------------------------
%----------------------------------------------------------------
%----------------------------------------------------------------
%----------------------------------------------------------------
%
\subsection{Rational points}
We denote by $\A^n(\fq)$ the $n$--dimensional $\fq$--vector space
$\fq^n$ and by $\Pp^n(\fq)$ the set of lines of the
$(n+1)$--dimensional $\fq$--vector space $\fq^{n+1}$. For a
projective variety $V\subset\Pp^n$ or an affine variety
$V\subset\A^n$, we denote by $V(\fq)$ the set of $\fq$--rational
points of $V$, namely $V(\fq):=V\cap \Pp^n(\fq)$ in the projective
case and $V(\fq):=V\cap \A^n(\fq)$ in the affine case.

For a projective variety $V$ of dimension $r$ and degree $\delta$ we
have the upper bound (see \cite[Proposition 12.1]{GhLa02} or
\cite[Proposition 3.1]{CaMa07})
\begin{equation}\label{eq: upper bound -- projective gral}
   |V(\fq)|\leq \delta p_r.
\end{equation}
On the other hand, if $V$ is an affine variety of dimension $r$ and
degree $\delta$, then (see, e.g., \cite[Lemma 2.1]{CaMa06})
\begin{equation}\label{eq: upper bound -- affine gral}
   |V(\fq)|\leq \delta q^r.
\end{equation}
%
%----------------------------------------------------------------------
%----------------------------------------------------------------------
%----------------------------------------------------------------------
%----------------------------------------------------------------------
%----------------------------------------------------------------------
%----------------------------------------------------------------------
%----------------------------------------------------------------------
%----------------------------------------------------------------------
%
\section{On the geometry of symmetric complete intersections}
\label{section: geometry symm complete inters}
Let $s,r,m$ be positive integers with $ m\leq s\leq r-m-2$. Let
$Y_1,\ldots,Y_s$ be indeterminates over $\cfq$ and let
$\fq[Y_1,\ldots,Y_s]$ be the ring of polynomials in $Y_1\klk Y_s$
and coefficients in $\fq$. We shall consider the weight $\wt$ on
$\fq[Y_1,\ldots,Y_s]$ defined by setting $\wt(Y_j):=j$ for $1\leq j
\leq s$. Let $S_1\klk S_m\in\fq[Y_1,\ldots,Y_s]$ and let $W_s\subset
\A^s$ be the affine $\fq$--variety that they define. Let
$(\partial\bfs{S}/
\partial \bfs{Y}):=(\partial S_i/
\partial Y_j)_{1\le i\le m,1\le j\le s}$ be the Jacobian matrix of
$S_{1}\klk S_{m}$ with respect to $Y_1\klk Y_s$. Assume that
$S_{1}\klk S_{m}$ satisfy the following conditions:
\begin{itemize}
  \item[($\sf H_1$)] $S_{1}\klk S_{m}$ form a regular sequence of $\fq[Y_1\klk
  Y_s]$;
  \item[($\sf H_2$)] $(\partial \bfs{S}/\partial \bfs{Y})(\bfs{y})$ has full rank $m$ for
every $\bfs{y}\in W_s$;
\item[($\sf H_3$)]
The components $S_1^{\wt}\klk S_m^{\wt}$ of highest weight of
$S_1\klk S_m$ satisfy ($\sf H_1$) and ($\sf H_2$).
\end{itemize}
A polynomial $F\in\fq[Y_1\klk Y_s]$ is called {\em weighted
homogeneous} (for the grading defined by the weight $\wt$) if all
the monomials arising in this dense representation have the same
weight. In this sense, it is clear that $S_1^{\wt}\klk S_m^{\wt}$
are weighted homogeneous.

In the next two remarks we show that polynomials $S_{1}\klk S_{m}$
as above, such that $S_{1}\klk S_{m}$ and $S_{1}^{\wt}\klk
S_{m}^{\wt}$ satisfy $(\sf H_2)$, necessarily satisfy $(\sf H_1)$
and $(\sf H_3)$. Nevertheless, as we shall rely repeatedly on the
hypotheses $(\sf H_1)$, $(\sf H_2)$ and $(\sf H_3)$, for the sake of
readability we express the conditions that the polynomials
$S_{1}\klk S_{m}$ must satisfy in this way.
\begin{remark}
If $S_1^{\wt}, \ldots, S_m^{\wt}$ satisfy $(\sf H_2)$, then
$S_1^{\wt}, \ldots, S_m^{\wt}$ form a regular sequence of
$\fq[Y_1\klk Y_s]$. Indeed, denote by $W^{\wt}\subset\A^s$ the
affine variety defined by $S_1^{\wt}, \ldots, S_m^{\wt}$ and let
$\mathcal{C}$ be an arbitrary absolutely irreducible component of
$W^{\wt}$. Then $\dim \mathcal{C}\ge s-m$. On the other hand, if
$\bfs y\in\mathcal{C}$ is a regular point of $W^{\wt}$, then the
fact that $S_1^{\wt}, \ldots, S_m^{\wt}$ satisfy $(\sf H_2)$ implies
that the tangent space $\mathcal{T}_{\bfs y}W^{\wt}$ of $W^{\wt}$ at
$\bfs y$ has dimension at most $s-m$. We conclude that
$\dim\mathcal{T}_{\bfs y}W^{\wt}=\dim\mathcal{C}=s-m$. In other
words, $W^{\wt}$ is of pure dimension $s-m$. Finally, as $S_1^{\wt},
\ldots, S_m^{\wt}$ are weighted homogeneous, the remark follows.
\end{remark}

\begin{remark}
If $S_{1}\klk S_{m}$ and $S_{1}^{\wt}\klk S_{m}^{\wt}$ satisfy $(\sf
H_2)$, then $S_{1}\klk S_{m}$ form a regular sequence of
$\fq[Y_1\klk Y_s]$. Indeed, let $S_{1}^{h_\wt}\klk S_{m}^{h_\wt}\in
\fq[Y_0,Y_1\klk Y_s]$ be the homogenizations of $S_1\klk S_m$ with
respect to the weight $\wt$. We claim that the affine variety
$V(S_1^{h_\wt}\klk S_m^{h_\wt})\subset\A^{s+1}$ is of pure dimension
$s-m+1$.

To show the claim, let $\mathcal{C}$ be an absolutely irreducible
component of $V(S_1^{h_\wt}\klk S_m^{h_\wt})$. It is clear that
$\dim \mathcal{C}\ge s-m+1$. Now let $\bfs y:=(y_0\klk y_s)\in
\mathcal{C}$ be a regular point of $V(S_1^{h_\wt}\klk S_m^{h_\wt})$.
Without loss of generality we may assume that either $y_0=1$ or
$y_0=0$. In the first case, the fact that the polynomials $S_{1}\klk
S_{m}$ satisfy $(\sf H_2)$ implies that $\dim\mathcal{T}_{\bfs
y}V(S_1^{h_\wt}\klk S_m^{h_\wt})\le s-m+1$. On the other hand, if
$y_0=0$, then $(y_1\klk y_s)\in W^{\wt}$, and since $S_{1}^{\wt}\klk
S_{m}^{\wt}$ satisfy $(\sf H_2)$, we deduce that
$\dim\mathcal{T}_{\bfs y}V(S_1^{h_\wt}\klk S_m^{h_\wt})\le s-m+1$.
In either case $\dim\mathcal{T}_{\bfs y}V(S_1^{h_\wt}\klk
S_m^{h_\wt})\le s-m+1$, which proves that $\dim\mathcal{C}=s-m+1$.
This finishes the proof of the claim.

Combining the claim with the fact that $S_1^{h_\wt}\klk S_m^{h_\wt}$
are weighted homogeneous, we conclude that $S_1^{h_\wt}\klk
S_m^{h_\wt}$ form a regular sequence. In particular,
$V(S_1^{h_\wt}\klk S_j^{h_\wt})$ is of pure dimension $s-j+1$ for
$1\le j\le m$. As the affine variety defined by the homogenization
of each element of the ideal $(S_1\klk S_j)$ with respect to the
grading defined by $\wt$ has dimension $\dim V(S_1\klk S_j)+1$ and
is contained in the affine variety defined by $S_1^{h_\wt}\klk
S_j^{h_\wt}$ for $1\le j\le m$, we have $\dim V(S_1\klk S_j)=s-j$
for $1\le j\le m$. This proves that $S_1\klk S_m$ form a regular
sequence.
\end{remark}
\begin{remark}
From ($\sf H_1$) we conclude that the variety $W_s\subset \A^s$
defined by $S_{1}\klk S_{m}$ is a set--theoretic complete
intersection of dimension $s-m$. Furthermore, by ($\sf H_2$) it
follows that the subvariety of $W_s$ defined by the set of common
zeros of the maximal minors of the Jacobian matrix $(\partial \bfs
S/\partial \bfs Y)$ has codimension at least one. Then \cite[Theorem
18.15]{Eisenbud95} proves that $S_1,\ldots,S_m$ define a radical
ideal, which implies that $W_s$ is a complete intersection.
\end{remark}

Let $X_1,\ldots,X_r$ be indeterminates over $\cfq$ and
$\fq[X_1,\ldots,X_r]$ the ring of polynomials in $X_1\klk X_r$ with
coefficients in $\fq$. Denote by $\Pi_1,\ldots,\Pi_s$ the first $s$
elementary symmetric polynomials of $\fq[X_1,\ldots,X_r]$. Let
$R_1,\ldots,R_m\in\fq[X_1,\ldots,X_r]$ be the polynomials
\begin{equation}
\label{def: polynomials in Ri} R_i:=S_i(\Pi_1,\ldots,\Pi_s) \quad
(1\leq i \leq m).
\end{equation}
Let $\deg R_i:=d_i$ for $1\le i\le m$. In what follows we shall
prove several facts concerning to the geometry of the affine
$\fq$--variety $V_r\subset\A^r$ defined by $R_1,\ldots,R_m$.

For this purpose, consider the following surjective morphism of
affine $\fq$--varieties:
\begin{align*}
{\bfs\Pi^r}: \A^r & \rightarrow  \A^r
   \\
   \bfs{x} & \mapsto  (\Pi_1(\bfs{x}),\ldots,\Pi_r(\bfs{x})).
\end{align*}
It is easy to see that  $\bfs\Pi^r$ is a  finite morphism. %(see e.g.,
%\cite[\S 5.3, Example 1]{Shafarevich94}).
In particular, the preimage $(\bfs\Pi^r)^{-1}(Z)$ of an irreducible
affine variety $Z\subset\A^r$ of dimension  $m$ is of pure dimension
$m$.

We now consider the polynomials $S_1,\ldots,S_m$ as elements of
$\fq[Y_1,\ldots,Y_r]$, and denote $W_r:=V(S_{1}\klk
S_m)\subset\A^r$. Observe that $V_r=(\bfs\Pi^r)^{-1}(W_r)$. Since
$S_1,\ldots,S_m$  form a regular
sequence of $\fq[Y_1,\ldots,Y_r]$, %the variety $W_r$ has pure
%dimension $r-m$. On the other hand,
the variety $W_r^j:=V(S_1,\ldots,S_j)\subset \A^r$ has pure
dimension $r-j$ for $1\leq j \leq m$. This implies that the variety
$V_r^j:=(\bfs\Pi^r)^{-1}(W_r^j)$ defined by $R_1,\ldots,R_j$ has
pure dimension $r-j$ for $1\leq j \leq m$. Hence, the polynomials
$R_{1}\klk R_{m}$ form a regular sequence of $\fq[X_1\klk X_r]$ and
we have the following result.
\begin{lemma}\label{lemma: V_r is complete inters}
Let $V_r\subset \A^r$ be the affine $\fq$--variety defined by
$R_{1}\klk R_{m}$. Then $V_r$ is a set--theoretic complete
intersection of dimension $r-m$.
\end{lemma}

Next we study the singular locus of $V_r$. To do this, we consider
the following morphism of $\fq$--varieties:
 \begin{align*}
\bfs \Pi : V_r & \rightarrow  W_s
   \\
   \bfs{x} & \mapsto  (\Pi_1(\bfs{x}),\ldots,\Pi_s(\bfs{x})).
 \end{align*}
For $\bfs{x}\in V_r$ and $\bfs{y}:=\bfs\Pi(\bfs{x})$, we denote by
$\mathcal{T}_{\bfs{x}}V_r$ and $\mathcal{T}_{\bfs{y}} W_s$ the
tangent spaces to $V_r$ at $\bfs{x}$ and to $W_s$ at $\bfs{y}$
respectively. We also consider the differential map of $\bfs\Pi$ at
$\bfs{x}$, namely
 \begin{align*}
   \mathrm{d}_{\bfs{x}}\bfs\Pi :  \mathcal{T}_{\bfs{x}}V_r &
   \rightarrow  \mathcal{T}_{\bfs{y}}W_s \\
   \bfs{v} & \mapsto  A(\bfs{x})\cdot \bfs{v},
 \end{align*}
where $A(\bfs{x})$ is the following  $(s\times r)$--matrix:
\begin{equation} \label{eq: Jacobian Pi}
    A(\bfs{x}):=\left(\frac{\partial\bfs\Pi}{\partial
    \bfs{X}}\right)(\bfs{x}):=\left(\frac{\partial \Pi_i}{\partial
    X_j}(\bfs{x})\right)_{1\leq i\leq s,\,  1 \leq j\leq r}.
\end{equation}

The main result of this section is an upper bound on the dimension
of the singular locus of $V_r$. To prove such a bound, we make some
remarks concerning the Jacobian matrix of the elementary symmetric
polynomials. It is well--known that the partial derivatives of the
elementary symmetric polynomials $\Pi_i$ satisfy the following
identities (see, e.g., \cite{LaPr02}) for $1\leq i,j \leq r$:
$$
    \frac{\partial \Pi_i}{\partial X_{j}}= \Pi_{i-1}-X_{j} \Pi_{i-2}
    + X_{j}^2 \Pi_{i-3} +\cdots+ (-1)^{i-1} X_{j}^{i-1}.
$$
As a consequence, if $A_r$ denotes the $(r\times r)$--Vandermonde
matrix
$$
    A_r:=(X_j^{i-1})_{1\leq i,j\leq r},
$$
then the Jacobian matrix $(\partial\bfs\Pi^r/\partial{\bfs X})$ of
$\bfs\Pi^r:=(\Pi_1,\dots,\Pi_r)$ with respect to $X_1,\dots,X_r$ can
be factored as follows:
\begin{equation} \label{eq: factorization Jacobian elem sim pols}
     \left(\frac{\partial {\bfs\Pi^r}}{\partial \bfs X}\right):=B_r\cdot
      A_r
       :=
        \left(
         \begin{array}{ccccc}
           1 & \ 0 & 0 &  \dots & 0
         \\
           \Pi_1 & - 1 & 0 &  &
          \\
           \Pi_2 & -\Pi_1 & 1 & \ddots & \vdots
          \\
           \vdots &\vdots  & \vdots & \ddots & 0
           \\
           \Pi_{r-1} & -\Pi_{r-2} & \Pi_{r-3} & \cdots &\!\! (-1)^{r-1}
         \end{array}
       \!\!\right)
     \cdot
         A_r.
  \end{equation}
Observe that $B_r$ is a square, lower--triangular matrix whose
determinant is equal to $(-1)^{{(r-1)r}/{2}}$. This implies that the
determinant of $(\partial{\bfs\Pi^r}/\partial{\bfs X})$ is equal, up
to a sign, to the determinant of $A_r$, namely,
$$\det \left(\frac{\partial {\bfs\Pi^r}}{\partial \bfs X}\right)
=(-1)^{{(r-1)r}/{2}} \prod_{1\le i < j\le r}(X_j-X_i).$$ Denote by
$(\partial \bfs{R}/\partial \bfs{X}):=(\partial R_i/\partial
X_j)_{1\le i\le m,1\leq j \le r}$ the Jacobian matrix of $R_{1}\klk
R_{m}$ with respect to  $X_1\klk X_r$. The following result
generalizes \cite[Theorem 3.2]{CeMaPePr14}.
\begin{theorem}\label{theorem: dimension singular locus V_r}
The set of points $\bfs{x}\in V_r$ for which $(\partial
\bfs{R}/\partial \bfs{X})(\bfs{x})$ does not have full rank, has
dimension at most $s-1$. In particular, the singular locus
$\Sigma_r$ of $V_r$ has dimension at most $s-1$.
\end{theorem}

\begin{proof}
By the chain rule, the partial derivatives of $R_i$ satisfy the
following equality:
$$
\left(\frac{\partial \bfs{R}}{\partial
\bfs{X}}\right)=\left(\frac{\partial \bfs{S}}{\partial
\bfs{Y}}\circ\bfs\Pi\right)\cdot
\left(\frac{\partial\bfs\Pi}{\partial \bfs{X}}\right).
$$
Fix an arbitrary point $\bfs{x}\in V_r$ for which $(\partial
\bfs{R}/\partial \bfs{X})(\bfs{x})$ does not have full rank. Let
$\bfs{v}\in\A^{m}$ be a nonzero vector in the left kernel of
$(\partial \bfs{R}/\partial \bfs{X})(\bfs{x})$. Thus
$$\bfs{0}=
\bfs{v}\cdot \left(\frac{\partial \bfs{R}}{\partial
\bfs{X}}\right)(\bfs{x})=\bfs{v}\cdot\left(\frac{\partial
\bfs{S}}{\partial \bfs{Y}}\right)\big(\bfs\Pi(\bfs{x})\big)\cdot
A(\bfs{x}),$$
where $A(\bfs{x})$ is the matrix defined in (\ref{eq: Jacobian Pi}).
Since by hypothesis ($\sf H_2$) the Jacobian matrix $(\partial
\bfs{S}/\partial \bfs{Y})\big(\bfs\Pi(\bfs{x})\big)$ has full rank,
we deduce that $\bfs{w}:=\bfs{v}\cdot\left({\partial
\bfs{S}}/{\partial \bfs{Y}}\right)\big(\Pi(\bfs{x})\big)\in\A^s$ is
a nonzero vector with $\bfs{w}\cdot A(\bfs{x})= \bfs{0}$. Hence, all
the maximal minors of  $A(\bfs{x})$ must be zero.

Observe that $A(\bfs{x})$ is the $(s\times r)$--submatrix of
$(\partial{\bfs\Pi^r}/\partial{\bfs X})(\bfs x)$ which consists of
the first $s$ rows of $(\partial\bfs \Pi^r/\partial{\bfs X})(\bfs
x)$. Therefore, according to (\ref{eq: factorization Jacobian elem
sim pols}) we conclude that
$$
A(\bfs{x})=B_{s,r}(\bfs{x})\cdot A_r(\bfs{x}),
$$
where $B_{s,r}(\bfs{x})$ is the $(s\times r)$--submatrix of
$B_r(\bfs{x})$  consisting of the first $s$ rows of $B_r(\bfs{x})$.
Since the last $r-s$ columns of $B_{s,r}(\bfs{x})$ are zero, we may
rewrite this identity as follows:
\begin{equation}\label{eq: factorization submatrix Jacobian elem symm}
A(\bfs{x})=B_s(\bfs{x})\cdot (x_j^{i-1})_{1\leq i \leq s, \, 1\leq j\leq r},
\end{equation}
where $B_s(\bfs{x})$ is the $(s \times s)$--submatrix of $B_r(\bfs{x})$ consisting of the first
 $s$ rows and the first $s$
columns of $B_r(\bfs{x})$.

Fix $1\leq l_1<\cdots<l_s\leq r$, set $I:=(l_1,\ldots,l_s)$ and consider
the $(s\times s)$--submatrix $M_I(\bfs{x})$ of $A(\bfs{x})$ consisting of the columns
 $l_1,\ldots,l_s$ of $A(\bfs{x})$, namely
$M_I(\bfs{x}):=({\partial \Pi_i}/\partial X_{l_{j}})_{1\leq i,j\leq
s}(\bfs{x})$. From (\ref{eq: factorization Jacobian elem sim pols})
and (\ref{eq: factorization submatrix Jacobian elem symm}) we deduce
that $M_I(\bfs{x})=B_s(\bfs{x})\cdot A_{s,I}(\bfs{x})$, where
$A_{s,I}(\bfs{x})$ is the Vandermonde matrix
$A_{s,I}(\bfs{x}):=(x_{l_{j}}^{i-1})_{1\leq i,j\leq s}$. As a
consequence,
\begin{equation}\label{eq: Vandermonde determinant}
\det\big(M_I(\bfs{x})\big) = (-1)^{\frac{(s-1)s}{2}}\det
A_{s,I}(\bfs{x}) = (-1)^{\frac{(s-1)s}{2}}\!\!\prod_{1\le m<n\le
s}\!\!(x_{l_{n}}-x_{l_{m}})=0.
\end{equation}

Since (\ref{eq: Vandermonde determinant}) holds for every
$I:=(l_1,\ldots,l_s)$ as above, we conclude that $\bfs{x}$ has at
least $s-1$ pairwise--distinct coordinates. In particular, the set
of points $\bfs{x}\in V_r$ for which $\mathrm{rank}(\partial
\bfs{R}/\partial \bfs{X})(\bfs{x})< m$ is contained in a finite
union of linear varieties of $\A^r$ of dimension $s-1$, and thus is
an affine variety of dimension at most $s-1$.

Now let $\bfs{x}$ be an arbitrary point of $\Sigma_r$. By Lemma
\ref{lemma: V_r is complete inters} we have that
 $\dim
\mathcal{T}_{\bfs{x}}V_r>r-m$. Thus, the rank of
 $\left({\partial \bfs{R}}/{\partial
\bfs{X}}\right)(\bfs{x})$ is less than $m$, for otherwise we would
have $\dim \mathcal{T}_{\bfs{x}}V_r\le r-m$, contradicting thus the
fact that $\bfs{x}$ is a singular point of $V_r$. This finishes the
proof of the theorem.
\end{proof}

From the proof of Theorem \ref{theorem: dimension singular locus
V_r} we conclude that the singular locus of $V_r$ is included in a
simple variety of ``low'' dimension.
\begin{remark}\label{remark: sigma_r included union linear vars}
Let notations and assumptions be as in Theorem \ref{theorem:
dimension singular locus V_r}. From the proof of Theorem
\ref{theorem: dimension singular locus V_r} we obtain the following
inclusion:
$$
\Sigma_r\subset \bigcup_{\mathcal{I}}\mathcal{L_{I}},
$$
where $\mathcal{I}:=\{I_1,\ldots,I_{s-1}\}$ runs over all the
partitions of $\{1,\ldots,r\}$ into $s-1$ nonempty subsets
$I_j\subset \{1,\ldots,r\}$ and $\mathcal{L_{I}}:=
\mathrm{span}(\bfs{v}^{I_1},\ldots,\bfs{v}^{I_{s-1}})$ is the linear
variety spanned by the vectors $\bfs{v}^{I_j}:=
(v_1^{I_j},\ldots,v_{r}^{I _j})$ defined by $v_m^{I_j}:=1$ for $m\in
I_j$ and $v_m^{I_j}:=0$ for $m\notin I_j$.
% In particular $\Sigma_r$ has dimension at most $s-1$.
\end{remark}

From Lemma \ref{lemma: V_r is complete inters} and  Theorem
\ref{theorem: dimension singular locus V_r} we obtain further
consequences concerning the polynomials $R_i$ and the variety $V_r$.
According to Theorem \ref{theorem: dimension singular locus V_r},
the set of points $\bfs{x}\in V_r$ for which the matrix $(\partial
\bfs{R}/\partial \bfs{X})(\bfs{x})$ does not have full rank, has
dimension at most $s-1$. Since $R_{1}\klk R_{m}$ form a regular
regular sequence and $s\leq r-m-2$, \cite[Theorem 18.15]{Eisenbud95}
shows that $R_{1}\klk R_{m}$ define a radical ideal of $\fq[X_1\klk
X_r]$, and thus $V_r$ is a complete intersection. Finally, the
B\'ezout inequality (\ref{eq: Bezout}) implies $\deg V_r\le
\prod_{i=1}^{m}d_i$. In other words, we have the following
statement.
\begin{corollary}\label{coro: radicality and degree Vr}
The polynomials $R_{1}\klk R_{m}$ define a radical ideal and the variety
$V_r$ is a complete intersection of degree at most $\deg V_r\le\prod_{i=1}^{m}d_i$.
\end{corollary}
%
%----------------------------------------------------------------------
%----------------------------------------------------------------------
%----------------------------------------------------------------------
%----------------------------------------------------------------------
%----------------------------------------------------------------------
%----------------------------------------------------------------------
%----------------------------------------------------------------------
%----------------------------------------------------------------------
%
\section{The number of rational points of symmetric complete intersections}
\label{section: estimates complete inters}
The results of the previous section on the geometry of the affine
$\fq$--variety $V_r\subset\A^r$, which is defined by the symmetric
polynomials $R_1,\ldots,R_m\in\fq[X_1,\ldots,X_r]$ of \eqref{def:
polynomials in Ri}, form the basis of our approach to estimate the
number of $\fq$--rational points of $V_r$. As we shall rely on
estimates for projective complete intersections defined over $\fq$,
we shall also need information on the behavior of $V_r$ ``at
infinity''.
%
%----------------------------------------------------------------------
%----------------------------------------------------------------------
%----------------------------------------------------------------------
%----------------------------------------------------------------------
%
\subsection{The geometry of the projective closure}
Consider the embedding of $\A^r$ into the projective space $\Pp^r$
which assigns to any $\bfs{x}:=(x_1,\ldots, x_r)\in\A^r$ the point
$(1:x_1:\dots:x_r)\in\Pp^r$. Then the closure
$\mathrm{pcl}(V_r)\subset\Pp^r$ of the image of $V_r$ under this
embedding in the Zariski topology of $\Pp^r$ is called the
projective closure of $V_r$. The points of $\mathrm{pcl}(V_r)$ lying
in the hyperplane $\{X_0=0\}$ are called the points of
$\mathrm{pcl}(V_r)$ at infinity.

It is well--known that $\mathrm{pcl} (V_r)$ is the $\fq$--variety of
$\mathbb{P}^r$ defined by the homogenization
$F^h\in\fq[X_0,\ldots,X_r]$ of each polynomial $F$ belonging to the
ideal $(R_{1}\klk R_{m})\subset\fq[X_1,\ldots,X_r]$ (see, e.g.,
\cite[\S I.5, Exercise 6]{Kunz85}). Denote by $(R_{1}\klk R_{m})^h$
the ideal generated by all the polynomials $F^h$ with $F\in
(R_{1}\klk R_{m})$. Since $(R_{1}\klk R_{m})$ is radical it turns
out that $(R_{1}\klk R_{m})^h$ is also a radical ideal (see, e.g.,
\cite[\S I.5, Exercise 6]{Kunz85}). Furthermore, $\mathrm{pcl}
(V_r)$ has pure dimension $r-m$ (see, e.g., \cite[Propositions
I.5.17 and II.4.1]{Kunz85}) and degree equal to $\deg V$ (see, e.g.,
\cite[Proposition 1.11]{CaGaHe91}).

Now we discuss the behavior of $\mathrm{pcl} (V_r)$ at infinity.
Consider the decomposition of each polynomial $R_i$ into its
homogeneous components, namely
$$R_i=R_i^{d_i}+R_i^{d_i-1}+\cdots+R_i^{0},$$
where each $R_i^j\in\fq[X_1\klk X_r]$ is homogeneous of degree $j$
or zero, $R_i^{d_i}$ being nonzero for $1\le j\le m$. Hence, the
homogenization of each $R_i$ is the polynomial
\begin{equation}\label{eq: Ri homogenization}
R_i^h=R_i^{d_i}+R_i^{d_i-1}X_0+\cdots+R_i^{0}X_0^{d_i}.
\end{equation}
It follows that $R_i^h(0,X_1\klk X_r)=R_i^{d_i}$ for $1\leq i \leq
m$. %Therefore, in order to analyze the singular locus of
%$\mathrm{pcl} (V_r)$ at infinity, we have to study the set of common
%zeros of the polynomials $R_1^{d_1},\ldots,R_m^{d_m}$.
Next we relate each $R_i^{d_i}$ with the component $S_i^{\wt}$ of
highest weight of $S_i$. Indeed, let $a_{i_1\klk i_s}Y_1^{i_1}\cdots
Y_s^{i_s}$ be an arbitrary monomial arising in the dense
representation of $S_i$. Then its weight $\wt(a_{i_1\klk
i_s}Y_1^{i_1}\cdots Y_s^{i_s})=\sum_{j=1}^s j \cdot i_j$ equals the
degree of the corresponding monomial $a_{i_1\klk i_s}\Pi_1^{i_1}
\cdots \Pi_s^{i_s}$ of $R_i$. Hence, we easily deduce the following
result.
\begin{lemma}\label{lemma: rel between R^di and S^wt}
Let $R_i^{d_i}$ be the homogeneous component of highest degree of
$R_i$ and $S_i^{\wt}$ the component of highest weight of $S_i$. Then
$R_i^{d_i}=S_i^{\wt}(\Pi_1,\ldots,\Pi_s)$ for $1\leq i \leq m$.
\end{lemma}

Let $\Sigma_{r}^{\infty}\subset\mathbb{P}^r$ be the singular locus
of $\mathrm{pcl}(V_r)$ at infinity, namely the set of singular
points of $\mathrm{pcl}(V_r)$ lying in the hyperplane $\{X_0=0\}$.
From Lemma \ref{lemma: rel between R^di and S^wt} we obtain critical
information about $\Sigma_{r}^{\infty}$.
\begin{lemma}\label{lemma: dim singular locus Vr at infinity}
The singular locus $\Sigma_{r}^{\infty}\subset\mathbb{P}^r$ at
infinity has dimension at most $s-2$.
\end{lemma}
\begin{proof}
Let $\bfs{x}:=(0:x_1:\dots: x_r)$ be an arbitrary point of
$\Sigma_{r}^{\infty}$. Since the polynomials $R_i^{h}$ vanish
identically in $\mathrm{pcl}(V_r)$, we have
$R_i^{h}(\bfs{x})=R_i^{d_i}(x_1\klk x_r)=0$ for $1\le i\le m$.

Denote by $(\partial \bfs{R}^{\bfs{d}}/\partial \bfs{X}):=(\partial
R_i^{d_i}/\partial X_j)_{1\le i\le m,1\leq j \le r}$ the Jacobian
matrix of $R_{1}^{d_1}\klk R_{m}^{d_m}$ with respect to  $X_1\klk
X_r$. We claim that $(\partial \bfs{R}^{\bfs{d}}/\partial
\bfs{X})(\bfs{x})$ does not have full rank. Indeed, if not, we would
have that $\dim \mathcal{T}_{\bfs{x}}(\mathrm{pcl}(V_r))\le r-m$,
which implies that $\bfs{x}$ is a nonsingular point of
$\mathrm{pcl}(V_r)$, contradicting thus the hypothesis on $\bfs{x}$.

Lemma \ref{lemma: rel between R^di and S^wt} asserts that
$R_i^{d_i}=S_i^{\wt}(\Pi_1,\ldots,\Pi_s)$ for $1\leq i \leq m$.
Therefore, combining the hypothesis ($\sf H_3$) with Lemma
\ref{lemma: rel between R^di and S^wt} we deduce that the
polynomials $R_1^{d_1}\klk R_m^{d_m}$ satisfy the hypotheses of
Theorem \ref{theorem: dimension singular locus V_r}. We conclude
that the set of points $\bfs x_{\mathrm{aff}}:=(0,x_1\klk x_r)\in
V(R_1^{d_1}\klk R_m^{d_m})\subset\A^r$ such that $(\partial
\bfs{R}^{\bfs{d}}/\partial \bfs{X})(\bfs x_{\mathrm{aff}})$ does not
have full rank, is an affine cone of $\A^{r+1}$ of dimension at most
$s-1$. It follows that the projective variety $\Sigma_{r}^{\infty}$
has dimension at most $s-2$.
\end{proof}

Our next result concerns the projective variety
$V(R_1^{d_1},\ldots,R_m^{d_m})\subset \Pp^{r-1}$. This will allow us
to obtain further information concerning to  the  behavior of
$\mathrm{pcl}(V_r)$ at infinity.

\begin{lemma}\label{Lemma: property variety Ri^di}
$V(R_1^{d_1},\ldots,R_m^{d_m})\subset \Pp^{r-1}$ is absolutely
irreducible of dimension $r-m-1$, degree at most
$\prod_{i=1}^{m}d_i$ and singular locus of dimension at most $s-2$.
\end{lemma}
\begin{proof}
%We first consider the  affine $\fq$--variety
%$V(R_1^{d_1},\ldots,R_m^{d_m})\subset \A^r$.
By Lemma \ref{lemma: rel between R^di and S^wt} we have
$R_{i}^{d_i}=S_{i}^{\wt}(\Pi_1,\ldots,\Pi_s)$ for $1\leq i \leq m$.
Since the polynomials $S_1^{\wt}\klk S_m^\wt$ satisfy the hypotheses
($\sf H_1$) and ($\sf H_2$), by Lemma \ref{lemma: V_r is complete
inters}, Theorem \ref{theorem: dimension singular locus V_r} and
Corollary \ref{coro: radicality and degree Vr} we see that the
affine $\fq$--variety of $\A^r$ defined by
$R_1^{d_1},\ldots,R_m^{d_m}$ is a cone of pure dimension $r-m$,
degree at most $\prod_{i=1}^{m}d_i$ and singular locus of dimension
at most $s-1$. Therefore, the projective variety
$V(R_1^{d_1},\ldots,R_m^{d_m})\subset \Pp^{r-1}$ has pure dimension
$r-m-1$, degree at most $\prod_{i=1}^{m}d_i$ and singular locus of
dimension at most $s-2$. In particular,
$V(R_1^{d_1},\ldots,R_m^{d_m})$ is a set--theoretic complete
intersection having a singular locus of codimension at least
$r-m-1-s+2\geq 3$. Then Theorem \ref{theorem: normal complete int
implies irred} shows that $V(R_1^{d_1},\ldots,R_m^{d_m})$ is
absolutely irreducible. This completes the proof of the lemma.
\end{proof}

Now we are able to establish all the facts we need on the geometry
of $\mathrm{pcl}(V_r)$ at infinity.
\begin{theorem}\label{th: pcl in infinity is absolutely irreducible}
$\mathrm{pcl}(V_r)\cap \{X_0=0\}\subset \Pp^{r-1}$ is a complete
intersection of dimension $r-m-1$, degree $\prod_{i=1}^{m}d_i$,
which is regular in codimension $r-m-s\ge 2$.
\end{theorem}
\begin{proof}
Recall that the projective variety $\mathrm{pcl}(V_r)$ has pure
dimension $r-m$. Hence, each irreducible component of
$\mathrm{pcl}(V_r)\cap \{X_0=0\}$ has dimension at least $r-m-1$.

From (\ref{eq: Ri homogenization}) we deduce that
$\mathrm{pcl}(V_r)\cap \{X_0=0\}\subset V(R_1^{d_1}\klk R_m^{d_m})$.
By Lemma \ref{Lemma: property variety Ri^di} we have that
$V(R_1^{d_1},\ldots,R_m^{d_m})$ is absolutely irreducible of
dimension $r-m-1$. It follows that $\mathrm{pcl}(V_r)\cap \{X_0=0\}$
is also absolutely irreducible of dimension $r-m-1$, and thus,
$$\mathrm{pcl}(V_r)\cap \{X_0=0\}=V(R_1^{d_1},\ldots,R_m^{d_m}).$$

According to Corollary \ref{coro: radicality and degree Vr}, the
polynomials $R_1^{d_1},\ldots,R_{m}^{d_m}$ define a radical ideal.
We conclude that $V(R_1^{d_1},\ldots,R_m^{d_m})$ is a complete
intersection and the B\'ezout theorem implies
$$\deg(\mathrm{pcl}(V_r)\cap \{X_0=0\})=\prod_{i=1}^{m}d_i.$$

It remains to show the assertion concerning the regularity of
$\mathrm{pcl}(V_r)\cap \{X_0=0\}$. From Lemma \ref{Lemma: property
variety Ri^di} we deduce that the singular locus of
$\mathrm{pcl}(V_r)\cap \{X_0=0\}$ has dimension at most $s-2$.
Therefore, $\mathrm{pcl}(V_r)\cap \{X_0=0\}$ is regular in
codimension $r-m-1-(s-2)-1$.
\end{proof}

We conclude this section with a statement that summarizes all the
facts we need concerning the geometry of the projective closure
$\mathrm{pcl}(V_r)$.
\begin{theorem}\label{theorem: proj closure of Vr is abs irred}
The projective variety $\mathrm{pcl}(V_r)\subset\Pp^r$ is a complete
intersection of dimension $r-m$ and degree $\prod_{i=1}^r d_i$,
which is regular in codimension $r-m-s\ge 2$.
\end{theorem}
\begin{proof}
We already know that $\mathrm{pcl}(V_r)$ is of pure dimension $r-m$.
According to Theorem \ref{theorem: dimension singular locus V_r},
the singular locus of $\mathrm{pcl}(V_r)$ lying in the open set
$\{X_0\not=0\}$ has dimension at most $s-1$. On the other hand, by
\cite[Lemma 1.1]{GhLa02a} the singular locus of $\mathrm{pcl}(V_r)$
at infinity is contained in the singular locus of
$\mathrm{pcl}(V_r)\cap \{X_0=0\}$, and Theorem \ref{th: pcl in
infinity is absolutely irreducible} shows that the singular locus of
$\mathrm{pcl}(V_r)\cap \{X_0=0\}$ has dimension at most $s-2$. We
conclude that the singular locus of $\mathrm{pcl}(V_r)$ has
dimension at most $s-1$ and thus $\mathrm{pcl}(V_r)$ is regular in
codimension $r-m-(s-1)-1=r-m-s$.

Observe that the following inclusions hold:
\begin{align*}
    V(R_1^{h},\ldots,R_m^{h})\cap\{X_0\not=0\}&
    \subset V(R_1,\ldots,R_m),\\
    V(R_1^{h},\ldots,R_m^{h})\cap\{X_0=0\}&\subset
    V(R_{1}^{d_1},\ldots,R_{m}^{d_m}).
\end{align*}
Lemma \ref{Lemma: property variety Ri^di} proves that
$V(R_1^{d_1},\ldots,R_m^{d_m})\subset \Pp^{r-1} $ is absolutely
irreducible of dimension $r-m-1$, while Lemma \ref{lemma: V_r is
complete inters} shows that $V(R_1,\ldots,R_m)\subset\A^r$ is of
pure dimension $r-m$. We conclude that
$V(R_1^{h},\ldots,R_m^{h})\subset\Pp^r$ has dimension at most $r-m$.
Since it is defined by $m$ polynomials, we deduce that it is a
set--theoretic complete intersection. Applying Theorem \ref{theorem:
dimension singular locus V_r} to $R_1\klk R_m$, and to
$R_1^{d_1}\klk R_m^{d_m}$, we see that the polynomials
$R_1^{h},\ldots,R_m^{h}$ define a radical ideal and the singular
locus of $V(R_1^{h},\ldots,R_m^{h})$ has codimension at least
$r-m-(s-1)\ge 3$. Then Theorem \ref{theorem: normal complete int
implies irred} proves that $V(R_1^{h},\ldots,R_m^{h})$ is absolutely
irreducible.

On the other hand, we observe that $\mathrm{pcl}(V_r)$ is contained
in the projective variety  $V(R_1^{h},\ldots,R_m^{h})$. Moreover,
since $\mathrm{pcl}(V_r)$ is of pure dimension $r-m$, it follows
that
\begin{equation}\label{id: eq de pcl}
\mathrm{pcl}(V_r)=V(R_1^{h},\ldots,R_m^{h}).
\end{equation}
Finally, since the polynomials $R_1^h,\ldots,R_m^h$ define a radical
ideal, \eqref{id: eq de pcl} and the B\'ezout theorem prove that
$\mathrm{pcl}(V_r)$ is a complete intersection of degree
$\prod_{i=1}^m d_i$.
\end{proof}
%
%----------------------------------------------------------------------
%----------------------------------------------------------------------
%----------------------------------------------------------------------
%----------------------------------------------------------------------
%----------------------------------------------------------------------
%----------------------------------------------------------------------
%----------------------------------------------------------------------
%----------------------------------------------------------------------
%
\subsection{Estimates on the number of $\fq$--rational points}
In what follows, we shall use an estimate on the number of
$\fq$--rational points of a projective complete intersection defined
over $\fq$ due to \cite{CaMaPr15} (see \cite{GhLa02a},
\cite{GhLa02}, \cite{CaMa07} and \cite{MaPePr16} for further
explicit estimates of this type). In \cite[Corollary 8.4]{CaMaPr15}
the authors prove that, for a complete intersection $V\subset\Pp^r$
defined over $\fq$, of dimension $r-m$ and multidegree $\bfs
d:=(d_1\klk d_m)$, which is regular in codimension $2$, the number
$|V(\fq)|$ of $\fq$--rational points of $V$ satisfies
\begin{equation}\label{eq: estimate rat points CML}
    \big||V(\fq)|-p_{r-m}\big|\leq 14 D^3\delta^2q^{r-m-1},
\end{equation}
where $p_{r-m}:=q^{r-m}+\cdots+ q+1=|\Pp^{r-m}(\fq)|$,
$\delta:=d_1\cdots d_m$ and $D:=\sum_{i=1}^m(d_i-1)$. Our aim is to
apply (\ref{eq: estimate rat points CML}) to estimate the number of
$\fq$--rational points of the variety $V_r$ defined by the
polynomials of \eqref{def: polynomials in Ri}.

Now let $V_r\subset\A^r$ be the affine $\fq$--variety defined by the
polynomials $R_1\klk R_m$ of \eqref{def: polynomials in Ri} and
$V_{r,\infty}:=\mathrm{pcl}(V_r)\cap\{X_0=0\}$. Combining Theorems
\ref{th: pcl in infinity is absolutely irreducible} and
\ref{theorem: proj closure of Vr is abs irred} with (\ref{eq:
estimate rat points CML}) we obtain
\begin{align*}
\big||\mathrm{pcl}(V_r)(\fq)|-p_{r-m}\big|&\le 14
D^3
\delta^2q^{r-m-1},  \\
\big||V_{r,\infty}(\fq)|-p_{r-m-1}\big|&\le 14 D^3
\delta^2q^{r-m-2},
\end{align*}
where $D:=\sum_{i=1}^m(d_i-1)$ and $\delta:=\prod_{i=1}^m d_i$.
As a consequence,
\begin{align*}
    \big||V_r(\fq)|-q^{r-m}\big|& =
    \big||\mathrm{pcl}(V_r)(\fq)|-|V_{r,\infty}(\fq)|-
      p_{r-m}+p_{r-m-1}\big|\nonumber\\[1ex]
      & \le
      \big||\mathrm{pcl}(V_r)(\fq)|-p_{r-m}\big|+
      \big||V_{r,\infty}(\fq)|-p_{r-m-1}\big|
      \nonumber\\[1ex] & \le  14 D^3
\delta^2(q+1)q^{r-m-2}.
  \end{align*}
As a consequence, we have the following result.
\begin{theorem}\label{theorem: nb point V_r}
Let $s,r,m$ be positive integers with $ m\leq s\leq r-m-2$. Let
$R_1,\ldots,R_m\in\fq[X_1,\ldots,X_r]$ be polynomials defined as
$R_i:=S_i(\Pi_1,\ldots,\Pi_s)$ for $1\leq i \leq m$, where $S_1\klk
S_m\in\fq[Y_1\klk Y_s]$ satisfy hypothesis $({\sf H_1})$, $({\sf
H_2})$ and $({\sf H_3})$. Denote $d_i:=\deg R_i$ for $1\leq i \leq
m$, $D:=\sum_{i=1}^m(d_i-1)$ and $\delta:=\prod_{i=1}^m d_i$. If
$V_r:=V(R_1\klk R_m)\subset\A^r$, then the following estimate holds:
$$\big||V_r(\fq)|-q^{r-m}\big|\le  14 D^3 \delta^2(q+1)q^{r-m-2}.$$
\end{theorem}

In the applications of the next sections not only estimates on the
number of $\fq$--rational points of a given complete intersection
are required, but also on the number of $\fq$--rational points with
certain pairwise--distinct coordinates, which is the subject of the
next results.
\begin{theorem}
\label{theorem: nb point V_r distinct coordinates} With notations
and assumptions as in Theorem \ref{theorem: nb point V_r}, given $i$
and $j$ with $1\le i<j\le r$ we have that $V_r\cap\{X_i=X_j\}$ is of
pure dimension $r-m-1$. In particular,
$$|V_r(\fq)\cap\{X_i=X_j\}|\le \delta q^{r-m-1}.$$
\end{theorem}
\begin{proof}
Theorem \ref{theorem: proj closure of Vr is abs irred} shows that
$\mathrm{pcl}(V_r)$ is a complete intersection which is regular in
codimension 2. Therefore, by Theorem \ref{theorem: normal complete
int implies irred} we conclude that it is absolutely irreducible.
This implies that $V_r$ is also absolutely irreducible.

Without loss of generality we may assume that $i=r-1$ and $j=r$.
Then $V_{r-1,r}:=V_r\cap \{X_{r-1}=X_r\}$ may be seen as the
subvariety of $\A^{r-1}$ defined by the polynomials
$R_1(\Pi_{1}^*,\ldots,\Pi_s^*)\klk R_m(\Pi_{1}^*,\ldots,\Pi_s^*)\in
\fq[X_1,\ldots,X_{r-1}]$, where $\Pi_i^*:=\Pi_{i}(X_1,\ldots,
X_{r-1},X_{r-1})$ is the polynomial obtained by substituting
$X_{r-1}$ for $X_r$ in the $i$th elementary symmetric polynomial
$\Pi_i$ of $\fq[X_1,\ldots, X_r]$. Observe that
\begin{equation}\label{eq: pi_i en x_(r-1) igual x_r}
\Pi_{i}^* = \Pi_i^{r-2} + 2X_{r-1}\cdot \Pi_{i-1}^{r-2} + X_{r-1}^2
\cdot \Pi_{i-2}^{r-2}
\end{equation}
where $\Pi_j^{r-2}$ is the $j$th elementary symmetric polynomial of
$\fq[X_1,\ldots, X_{r-2}]$ for $1\le j\le s$.

Let $\bfs \Pi^*:=(\Pi_1^*\klk\Pi_s^*)$ and denote by
$(\partial\bfs\Pi^*/\partial\bfs X^*)$ the Jacobian matrix
$\bfs\Pi^*$ with respect to $X_1\klk X_{r-1}$. We observe that the
set of points $\bfs x$ of $V_{r-1,r}$ for which the Jacobian matrix
$\big(\partial(\bfs R\circ\bfs\Pi^*)/\partial\bfs X^*\big)(\bfs x)$
does not have maximal rank, has dimension at most $s$. Indeed, from
(\ref{eq: pi_i en x_(r-1) igual x_r}) we conclude that the nonzero
$(s\times s)$--minor of the Jacobian matrix
$(\partial{\bfs\Pi^*}/\partial\bfs X^*)$ determined by any choice
$i_1\klk i_s$ of columns with $1 \leq i_1 < i_2 < \cdots < i_s \leq
r-2$ equals the corresponding nonzero minor of
$(\partial{\Pi_{i}^{r-2}}/\partial{X_j})_{1\le i\le s,1\le j\le
r-1}$. In particular, each such nonzero maximal minor of
$(\partial{\Pi_{i}^*}/\partial{X_j})_{1\le i\le s,1\le j\le r-2}$
is, up to a sign, a Vandermonde determinant depending on $s$ of the
indeterminates $X_1, \ldots, X_{r-2}$. Arguing as in the proof of
Theorem \ref{theorem: dimension singular locus V_r} we deduce that
the set of points $\bfs x$ of $V_{r-1,r}$ for which the Jacobian
matrix $\big(\partial(\bfs R\circ\bfs\Pi^*)/\partial\bfs
X^*\big)(\bfs x)$ does not have maximal rank, is included in a union
of linear varieties of dimension $s$ (see Remark \ref{remark:
sigma_r included union linear vars}), and hence it has dimension at
most $s$.

Let $\mathcal{C}$ be any irreducible component of $V_{r-1,r}$. Then
$\mathcal{C}$ has dimension at least $r-m-1$. Since $r-m-1-s\ge 1$,
for a generic point $\bfs x\in\mathcal{C}$ the Jacobian matrix
$\big(\partial(\bfs R\circ\bfs\Pi^*)/\partial\bfs X^*\big)(\bfs x)$
has maximal rank $m$. We conclude that the tangent space of
$V_{r-1,r}$ at $\bfs x$ has dimension at most $r-m-1$, which implies
that $\mathcal{C}$ has dimension $r-m-1$. As a consequence,
$V_{r-1,r}$ is of pure dimension $r-m-1$, finishing thus the proof
of the first assertion.

On the other hand, from the B\'ezout inequality (\ref{eq: Bezout})
it follows that $ \deg V_r\cap\{X_i=X_j\} \le\deg V_r$. Then the
second assertion readily follows from \eqref{eq: upper bound --
affine gral}.
\end{proof}
Let $\mathcal{I}$ be a subset of the set $\{(i,j):1\le i<j\le r\}$
and $V_r^{=}\subset\A^r$ the variety defined as
$$V_r^{=}:=\bigcup_{(i,j)\in\mathcal{I}}V_r\cap\{X_i=X_j\}.$$
%
%Observe that $V_{r}^==V_{r}\cap \mathcal{H}_r$, where
%$\mathcal{H}_r\subset \A^r$ is the hypersurface defined by the
%polynomial $F_r:=\prod_{1 \leq i <j\leq r} (X_i-X_j)$.
Finally, denote $V_r^{\not=}:=V_r\setminus V_r^=$. We have the
following result.
\begin{corollary}
\label{coro: nb point V_r distinct coordinates} With notations and
assumptions as in Theorem \ref{theorem: nb point V_r}, we have
$$\big||V_r^{\not=}(\fq)|-q^{r-m}\big|\le  14 D^3 \delta^2(q+1)q^{r-m-2}
+|\mathcal{I}|\delta q^{r-m-1}.$$
\end{corollary}
\begin{proof}
According to Theorem \ref{theorem: nb point V_r distinct
coordinates},
$$|V_r^=(\fq)|\le \sum_{(i,j)\in\mathcal{I}}\delta q^{r-m-1}
\le |\mathcal{I}|\delta q^{r-m-1}.$$
Therefore, by Theorem \ref{theorem: nb point V_r} it follows that
\begin{align*}
\big||V_r^{\not=}(\fq)|-q^{r-m}\big|&\le
\big||V_r(\fq)|-q^{r-m}\big|+ |V_r^{=}(\fq)|\\
&\le 14 D^3 \delta^2(q+1)q^{r-m-2}+|\mathcal{I}|\delta q^{r-m-1}.
\end{align*}
This finishes the proof of the corollary.
\end{proof}
%
%----------------------------------------------------------------------
%----------------------------------------------------------------------
%----------------------------------------------------------------------
%----------------------------------------------------------------------
%----------------------------------------------------------------------
%----------------------------------------------------------------------
%----------------------------------------------------------------------
%----------------------------------------------------------------------
%
%----------------------------------------------------------------------
%----------------------------------------------------------------------
%----------------------------------------------------------------------
%----------------------------------------------------------------------
%
\section{The distribution of factorization patterns}
\label{section: distribution of fact patterns}
This section is devoted to the first application of the framework of
the previous section, namely we obtain an estimate on the number
$|\mathcal{A}_{\bfs\lambda}|$ of elements on a linear family
$\mathcal{A}$ of monic polynomials of $\fq[T]$ of degree $n$ having
factorization pattern $\bfs\lambda:=1^{\lambda_1}2^{\lambda_2}\cdots
n^{\lambda_n}$. Our estimate asserts that
$|\mathcal{A}_{\bfs\lambda}|=
\mathcal{T}(\bfs\lambda)\,q^{n-m}+\mathcal{O}(q^{n-m-1})$, where
$\mathcal{T}(\bfs\lambda)$ is the proportion of elements of the
symmetric group of $n$ elements with cycle pattern $\bfs\lambda$ and
$m$ is the codimension of $\mathcal{A}$. To this end, we reduce the
question to estimate the number of $\fq$--rational points with
pairwise--distinct coordinates of a certain family of complete
intersections defined over $\fq$. These complete intersections are
defined by symmetric polynomials which satisfy the hypotheses of
Theorem \ref{theorem: nb point V_r} and Corollary \ref{coro: nb
point V_r distinct coordinates}.

Let $T$ be an indeterminate over $\cfq$. For a positive integer $n$,
let $\mathcal{P}:={\mathcal P}_n$ be the set of all monic
polynomials in $\fq[T]$ of degree $n$. Let
$\lambda_1,\dots,\lambda_n$ be nonnegative integers such that
$$\lambda_1+2\lambda_2+\cdots+n\lambda_n=n.$$
We denote by ${\mathcal P}_{\bfs \lambda}$ the set of $f\in
\mathcal{P}$ with factorization pattern $\bfs
\lambda:=1^{\lambda_1}2^{\lambda_2}\cdots n^{\lambda_n}$, namely the
elements $f\in \mathcal{P}$ having exactly $\lambda_i$ monic
irreducible factors over $\fq$ of degree $i$ (counted with
multiplicity) for $1\le i\le n$. Further, for any subset
$\mathcal{S}\subset\mathcal{P}$ we shall denote
$\mathcal{S}_{\bfs{\lambda}}:=\mathcal{S}\cap\mathcal{P}_{\bfs
\lambda}$.

In \cite{Cohen70},  S. Cohen showed that the proportion of elements
of $\mathcal{P}_{\bfs\lambda}$ in $\mathcal{P}$ is roughly the
proportion $\mathcal{T}({\bfs{\lambda}})$ of permutations with cycle
pattern $\bfs \lambda$ in the $n$th symmetric group $\mathbb{S}_n$,
where a permutation of $\mathbb{S}_n$ is said to have cycle pattern
$\bfs\lambda$ if it has exactly $\lambda_i$ cycles of length $i$ for
$1\le i\le n$. More precisely, Cohen proved that
$$%\begin{equation}\label{eq: intro: Cohen}
|\mathcal{P}_{\bfs\lambda}|=\mathcal{T}(\bfs\lambda)\,q^n+
\mathcal{O}(q^{n-1}),
$$%\end{equation}
where the constant underlying the $\mathcal{O}$--notation depends
only on $\bfs \lambda$. Observe that the number of permutations in
$\mathbb{S}_n$ with cycle pattern $\bfs\lambda$ is $n!/w({\bfs
\lambda})$, where
$$w({\bfs \lambda}):=1^{\lambda_1}2^{\lambda_2}\dots n^{\lambda_n}
\lambda_1!\lambda_2!\dots\lambda_n!.$$
In particular, $\mathcal{T}({\bfs{\lambda}})={1}/{w({\bfs
\lambda})}$.

Further, in \cite{Cohen72} Cohen called
$\mathcal{S}\subset\mathcal{P}$ {\em uniformly distributed} if the
proportion $|\mathcal{S}_{\bfs\lambda}|/|\mathcal{S}|$ is roughly
$\mathcal{T}(\bfs\lambda)$ for every factorization pattern
$\bfs\lambda$. The main result of this paper (\cite[Theorem
3]{Cohen72}) provides a criterion for a linear family $\mathcal{S}$
of polynomials of $\mathcal{P}$ to be uniformly distributed in the
sense above. For any such linear family $\mathcal{S}$ of codimension
$m \le n-2$, assuming that the characteristic $p$ of $\fq$ is
greater than $n$, it is shown that
\begin{equation}\label{eq: intro: Cohen72}
|\mathcal{S}_{\bfs\lambda}|=\mathcal{T}(\bfs\lambda)\,q^{n-m}+
\mathcal{O}(q^{n-m-\frac{1}{2}}).
\end{equation}
A difficulty with \cite[Theorem 3]{Cohen72} is that the hypotheses
for a linear family of $\mathcal{P}$ to be uniformly distributed
seem complicated and not easy to verify. We would also like to
improve the asymptotic behavior of the error term
$\mathcal{O}(q^{n-m-{1}/{2}})$ and remove restrictions on the
characteristic of $\fq$. Finally, we are interested in explicit
estimates, that is, an explicit admissible expression for the
$\mathcal{O}$--constant in (\ref{eq: intro: Cohen72}).

For this purpose, we shall consider the linear families in
$\mathcal{P}$ that we now describe. Let $m$, $s$ be positive
integers with $q >n$ and $m \leq s\leq n-m-2$, let $A_{n-s},\ldots,
A_{n-1}$ be indeterminates over $\cfq$ and let $L_1,\ldots, L_m$ be
linear forms of $\fq[A_ {n-s},\ldots,A_{n-1}]$ which are linearly
independent. For $\bfs \alpha:=(\alpha_1,\ldots,\alpha_m)\in \fq^m$,
we set $\bfs{L}:=(L_1,\ldots,L_m)$ and we consider the linear
variety $\mathcal{A}:=\mathcal{A}(\bfs{L},\bfs\alpha)$ defined in
the following way:
\begin{equation}\label{eq: intro: definition A}
\mathcal{A}:=\left\{T^n+a_{n-1}T^{n-1}\plp a_0\in\mathcal{P}: \bfs
L(a_{n-s}\klk a_{n-1})+\bfs \alpha=\bfs 0\right\}.
\end{equation}
We may assume without loss of generality that the Jacobian matrix
$(\partial \bfs L /\partial \bfs A)$ is lower triangular in row
echelon form and denote by $1 \leq i_1 < \ldots < i_m \leq s$ the
positions corresponding to the pivots. Given a factorization pattern
$\bfs \lambda:= 1^{\lambda_1}\ldots n^{\lambda_n}$, our goal is to
estimate the number $|\mathcal{A}_{\bfs\lambda}|$ of elements in
$\mathcal{A}$ with factorization pattern $\bfs \lambda$. %We shall
%show that any such family $\mathcal{A}$ is uniformly distributed.
%
%More precisely, we have the following result.
%
%\begin{theorem}\label{theorem: intro: main}
%    Let $\mathcal{A}_{\bfs \lambda}:=\mathcal{A}
%    \cap\mathcal{P}_{\bfs\lambda}$. If $q>n$ and $m\le s\le n-m-2$,  then we have
%    %
%    \begin{equation}\label{eq: intro: estimate FP II}
%    \big||\mathcal{A}_{\bfs
%        \lambda}|-\mathcal{T}(\bfs\lambda)\,q^{n-m}\big|\le
%    q^{n-m-1}\big(21\,\mathcal{T}(\bfs\lambda)\,D_{\bfs L}^3\delta_{\bfs
%        L}^2 + n^2\big).
%    \end{equation}
%    %
%\end{theorem}
%We shall see that $\delta_{\bfs L}$ and $D_{\bfs L}$ are certain
%explicit discrete invariants associated with $\bfs L$.
%
%----------------------------------------------------------------------
%----------------------------------------------------------------------
%
\subsection{Factorization patterns and roots}
Let $\mathcal{A}\subset\mathcal{P}:=\mathcal{P}_n$ be the linear
family of \eqref{eq: intro: definition A} and
$\bfs\lambda:=1^{\lambda_1}\cdots n^{\lambda_n}$ a factorization
pattern. Following the approach of \cite{CeMaPe15}, we shall show
that the condition that an element of $\mathcal{A}$ has
factorization pattern $\boldsymbol{\lambda}$ can be expressed in
terms of certain elementary symmetric polynomials.
%We shall show that $|\mathcal{A}_{\bfs\lambda}|$
%can be expressed in terms of the number of common $\fq$--rational
%zeros of certain polynomials $R_1\klk R_m\in\fq[X_1\klk X_n]$.

Let $f$ be an element of $\mathcal{P}$ and $g\in \fq[T]$ a monic
irreducible factor of $f$ of degree $i$. Then $g$ is the minimal
polynomial of a root $\alpha$ of $f$ with $\fq(\alpha)=\fqi$. Denote
by $\mathbb G_i$ the Galois group $\mbox{Gal}(\fqi,\fq)$ of $\fqi$
over $\fq$. We may express $g$ in the following way:
$$g=\prod_{\sigma\in\mathbb G_i}(T-\sigma(\alpha)).$$
Hence, each irreducible factor $g$ of $f$ is uniquely determined by
a root $\alpha$ of $f$ (and its orbit under the action of the Galois
group of $\cfq$ over $\fq$), and this root belongs to a field
extension of $\fq$ of degree $\deg g$. Now, for
$f\in\mathcal{P}_{\bfs \lambda}$, there are $\lambda_1$ roots of $f$
in $\fq$, say $\alpha_1,\dots,\alpha_{\lambda_1}$ (counted with
multiplicity), which are associated with the irreducible factors of
$f$ in $\fq[T]$ of degree 1; we may choose $\lambda_2$ roots of $f$
in $\fqtwo\setminus\fq$ (counted with multiplicity), say
$\alpha_{\lambda_1+1},\dots, \alpha_{\lambda_1+\lambda_2}$, which
are associated with the $\lambda_2$ irreducible factors of $f$ of
degree 2, and so on. From now on we shall assume that a choice of
$\lambda_1\plp\lambda_n$ roots $\alpha_1\klk\alpha_{\lambda_1
\plp\lambda_n}$ of $f$ in $\cfq$ is made in such a way that each
monic irreducible factor of $f$ in $\fq[T]$ is associated with one
and only one of these roots.

Our aim is to express the factorization of $f$ into irreducible
factors in $\fq[T]$ in terms of the coordinates of the chosen
$\lambda_1\plp \lambda_n$ roots of $f$ with respect to certain bases
of the corresponding extensions $\fq\hookrightarrow\fqi$ as
$\fq$--vector spaces. To this end, we express the root associated
with each irreducible factor of $f$ of degree $i$ in a normal basis
$\Theta_i$ of the field extension $\fq\hookrightarrow \fqi$.

Let $\theta_i\in \fqi$ be a normal element and $\Theta_i$ the normal
basis of the extension $\fq\hookrightarrow\fqi$ generated by
$\theta_i$, i.e.,
$$\Theta_i=\left \{\theta_i,\cdots, \theta_i^{q^{i-1}}\right\}.$$
Observe that the Galois group $\mathbb G_i$ is cyclic and the
Frobenius map $\sigma_i:\fqi\to\fqi$, $\sigma_i(x):=x^q$ is a
generator of $\mathbb{G}_i$. Thus, the coordinates in the basis
$\Theta_i$ of all the elements in the orbit of a root
$\alpha_k\in\fqi$ of an irreducible factor of $f$ of degree $i$ are
the cyclic permutations of the coordinates of $\alpha_k$ in the
basis $\Theta_i$.

%The coordinate vector  of each element $\alpha $ of $\fqi$ in   base
%$\Theta_i$ is a vector  in $\fq^i$.
The vector that gathers the coordinates of all the roots
$\alpha_1\klk\alpha_{\lambda_1+\dots+\lambda_n}$ we chose to
represent the irreducible factors of $f$ in the normal bases
$\Theta_1\klk \Theta_n$ is an element of $\fq^n$, which is denoted
by ${\bfs x}:=(x_1,\dots,x_n)$. Set
\begin{equation}\label{eq: fact patterns: ell_ij}
\ell_{i,j}:=\sum_{k=1}^{i-1}k\lambda_k+(j-1)\,i
\end{equation}
for $1\le j \le \lambda_i$ and $1\le i \le n$. Observe that the
vector of coordinates of a root
$\alpha_{\lambda_1\plp\lambda_{i-1}+j}\in\fqi$ is the sub-array
$(x_{\ell_{i,j}+1},\dots,x_{\ell_{i,j}+i})$ of $\bfs x$. With these
notations, the $\lambda_i$ irreducible factors of $f$ of degree $i$
are the polynomials
\begin{equation}\label{eq: fact patterns: gij}g_{i,j}=\prod_{\sigma\in\mathbb G_i}
\Big(T-\big(x_{\ell_{i,j}+1}\sigma(\theta_i)+\dots+
x_{\ell_{i,j}+i}\sigma(\theta_i^{q^{i-1}})\big)\Big)
\end{equation}
for $1\le j \le \lambda_i$. In particular,
\begin{equation}\label{eq: fact patterns: f factored with g_ij}
f=\prod_{i=1}^n\prod_{j=1}^{\lambda_i}g_{i,j}.
\end{equation}

Let $X_1\klk X_n$ be indeterminates over $\cfq$, set $\bfs
X:=(X_1,\dots,X_n)$ and consider the polynomial $G\in
\fq[\bfs{X},T]$ defined as
\begin{equation}\label{eq: fact patterns: pol G}
G:=\prod_{i=1}^n\prod_{j=1}^{\lambda_i}G_{i,j},\quad
G_{i,j}:=\prod_{\sigma\in\mathbb G_i}
\Big(T-\big(X_{\ell_{i,j}+1}\sigma(\theta_i)+
\dots+X_{\ell_{i,j}+i}\sigma(\theta_i^{q^{i-1}})\big)\Big),
\end{equation}
where the $\ell_{i,j}$ are defined as in \eqref{eq: fact patterns:
    ell_ij}. Our previous arguments show that $f\in\mathcal{P}$ has
factorization pattern ${\bfs \lambda}$ if and only if there exists
$\bfs x\in\fq^n$ with $f=G({\bfs x},T)$.

Next we discuss how many elements $\bfs x\in\fq^n$ yield an
arbitrary polynomial $f=G(\bfs x,T)\in\mathcal{P}_{\bfs\lambda}$.
For $\alpha\in\fqi$, we have $\fq(\alpha)=\fqi$ if and only if its
orbit under the action of the Galois group $\G_i$ has exactly $i$
elements. In particular, if $\alpha$ is expressed by its coordinate
vector $\bfs x\in\fq^i$ in the normal basis $\Theta_i$, then the
coordinate vectors of the elements of the orbit of $\alpha$ form a
cycle of length $i$, because the Frobenius map $\sigma_i\in{\mathbb
    G}_i$ permutes cyclically the coordinates. As a consequence, there
is a bijection between cycles of length $i$ in $\fq^i$ and elements
$\alpha\in\fqi$ with $\fq(\alpha)=\fqi$.

To make this relation more precise, we introduce the notion of an
array of type $\bfs \lambda$. Let $\ell_{i,j}$ $(1\le i\le n,\ 1\le
j\le\lambda_i)$ be defined as in \eqref{eq: fact patterns: ell_ij}.
We say that ${\bfs x}=(x_1,\dots, x_n)\in\fq^n$ is of {\em type
$\bfs \lambda$} if and only if each sub-array $\bfs
x_{i,j}:=(x_{\ell_{i,j}+1},\dots,x_{\ell_{i,j}+i})$ is a cycle of
length $i$. The following result relates the quantity
$\mathcal{P}_{\boldsymbol{\lambda}}$ with the set of elements of
$\fq^n$ of type $\bfs \lambda$. The proof of this result is only
sketched here (see \cite[Lemma 2.2]{CeMaPe15} for a full proof).
\begin{lemma}
    \label{lemma: fact patterns: G(x,T) with fact pat lambda}

 For any
    ${\bfs x}=(x_1,\dots, x_n)\in \fq^n$, the polynomial $f:=G({\bfs x},T)$
    has factorization pattern $\bfs \lambda$ if and only if ${\bfs x}$
    is of type $\bfs \lambda$. Furthermore, for each square--free
    polynomial $f\in \mathcal{P}_{\bfs \lambda}$ there are $w({\bfs
        \lambda}):=\prod_{i=1}^n i^{\lambda_i}\lambda_i!$ different ${\bfs
        x}\in \fq^n $ with $f=G({\bfs x},T)$.
\end{lemma}

\begin{proof}[Sketch of proof]
    Let $\Theta_1,\dots,\Theta_n$ be the normal bases introduced before.
    Each $\bfs x\in \fq^n$ is associated with a unique finite sequence
    of elements $\alpha_k$ $(1\le k\le \lambda_1+\dots+\lambda_n)$ as
    follows: each $\alpha_{\lambda_1\plp\lambda_{i-1}+j}$ with $1\le
    j\le\lambda_i$ is the element of $\fqi$ whose coordinate vector in
    the basis $\Theta_i$ is the sub-array $(x_{\ell_{i,j}+1},
    \dots,x_{\ell_{i,j}+i})$ of $\bfs x$.
    Suppose that $G({\bfs x},T)$ has factorization pattern $\bfs
    \lambda$ for a given $\bfs x\in\fq^n$. Fix $(i,j)$ with $1\le i\le
    n$ and $1\le j\le\lambda_i$. Then $G({\bfs x},T)$ is factored as in
    \eqref{eq: fact patterns: gij}--\eqref{eq: fact patterns: f factored
        with g_ij}, where each $g_{i,j}\in\fq[T]$ is irreducible, and hence
    $\fq(\alpha_{\lambda_1\plp\lambda_{i-1}+j})=\fqi$. We conclude that
    the sub-array $(x_{\ell_{i,j}+1}, \dots,x_{\ell_{i,j}+i})$ defining
    $\alpha_{\lambda_1\plp\lambda_{i-1}+j}$  is a cycle of length $i$.
    This proves that $\bfs x$ is of type $\bfs\lambda$.
%
%    On the other hand, given $\bfs x\in\fq^n$ of type $\bfs \lambda$,
%    fix $(i,j)$ with $1\le i\le n$ and $1\le j\le\lambda_i$. Then
%    $\fq(\alpha_{\lambda_1\plp\lambda_{i-1}+j})=\fqi$, because the
%    sub-array $(x_{\ell_{i,j}+1}, \dots,x_{\ell_{i,j}+i})$ is a cycle of
%    length $i$ and thus the orbit of
%    $\alpha_{\lambda_1\plp\lambda_{i-1}+j}$ under the action of $\mathbb
%    G_i$ has $i$ elements. This implies that the factor $g_{i,j}$ of
%    $G({\bfs x},T)$ defined as in (\ref{eq: fact patterns: gij}) is
%    irreducible of degree $i$. We deduce that $f:=G(\bfs x,T)$ has
%    factorization pattern $\bfs\lambda$.

    Furthermore, for $\bfs x\in\fq^n$ of type $\bfs\lambda$, the
    polynomial $f:=G(\bfs x,T)\in\mathcal{P}_{\bfs\lambda}$ is
    square--free if and only if all the roots
    $\alpha_{\lambda_1\plp\lambda_{i-1}+j}$ with $1\le j\le \lambda_i$
    are pairwise--distinct, non--conjugated elements of $\fqi$. This
    implies that no cyclic permutation of a sub-array
    $(x_{\ell_{i,j}+1}, \dots,x_{\ell_{i,j}+i})$ with $1\le
    j\le\lambda_i$ agrees with another cyclic permutation of another
    sub-array $(x_{\ell_{i,j'}+1}, \dots,x_{\ell_{i,j'}+i})$. As cyclic
    permutations of any of these sub-arrays and permutations of these
    sub-arrays yield elements of $\fq^n$ associated with the same
    polynomial $f$, there are $w({\bfs
        \lambda}):=\prod_{i=1}^n i^{\lambda_i}\lambda_i!$ different elements
    $\bfs x\in\fq^n$ with $f=G(\bfs x,T)$.
\end{proof}

Consider the polynomial $G$ defined in \eqref{eq: fact patterns: pol
G} as an element of $\fq[\bfs X][T]$. We shall express the
coefficients of $G$ by means of the vector of linear forms $\bfs
Y:=(Y_1\klk Y_n)$, with $Y_i\in\cfq[\bfs X]$ for $1\le i\le n$,
defined in the following way:
\begin{equation}\label{eq: fact patterns: def linear forms Y}
(Y_{\ell_{i,j}+1},\dots,Y_{\ell_{i,j}+i})^{t}:=A_{i}\cdot
(X_{\ell_{i,j}+1},\dots, X_{\ell_{i,j}+i})^{t} \quad(1\le j\le
\lambda_i,\ 1\le i\le n),
\end{equation}
where $A_i\in\fqi^{i\times i}$ is the matrix
$$A_i:=\left(\sigma(\theta_i^{q^{h}})\right)_{\sigma\in {\mathbb G}_i,\, 0\le h\le i-1}.$$
According to (\ref{eq: fact patterns: pol G}), we may express the
polynomial $G$ as
$$G=\prod_{i=1}^n\prod_{j=1}^{\lambda_i}\prod_{k=1}^i(T-Y_{\ell_{i,j}+k})=
\prod_{k=1}^n(T-Y_k)=T^n+\sum_{k=1}^n(-1)^k\,(\Pi_k(\bfs Y))\,
T^{n-k},$$
where $\Pi_1(\bfs Y)\klk \Pi_n(\bfs Y)$ are the elementary symmetric
polynomials of $\fq[\bfs Y]$. By the expression of $G$ in \eqref{eq:
    fact patterns: pol G} we deduce that $G$ belongs to $\fq[{\bfs
    X},T]$, which in particular implies that $\Pi_k(\bfs Y)$ belongs to
$\fq[{\bfs X}]$ for $1\le k\le n$. Combining these arguments with
Lemma \ref{lemma: fact patterns: G(x,T) with fact pat lambda} we
obtain the following result.
\begin{lemma}
    \label{lemma: fact patterns: sym pols and pattern lambda} A
    polynomial $f:=T^n+a_{n-1}T^{n-1}\plp a_0\in\mathcal{P}$ has
    factorization pattern $\bfs \lambda$ if and only if there exists
    $\bfs{x}\in\fq^n$ of type $\bfs \lambda$ such that
    \begin{equation}\label{eq: fact patterns: sym pols and pattern lambda}
    a_k= (-1)^{{n-k}}\,\Pi_{n-k}(\bfs Y(\bfs x)) \quad(0\le k\le n-1).
    \end{equation}
    In particular, for $f$ square--free, there are $w(\bfs \lambda)$
    elements $\bfs x$ for which (\ref{eq: fact patterns: sym pols and
        pattern lambda}) holds.
\end{lemma}

As a consequence, we may express the condition that an element of
$\mathcal{A}:=\mathcal{A}(\bfs L,\bfs\alpha)$ has factorization
pattern $\bfs \lambda$ in terms of the elementary symmetric
polynomials $\Pi_1\klk\Pi_{s}$ of $\fq[\bfs Y]$.
\begin{corollary}\label{coro: fact patterns: systems pattern lambda}
    A polynomial $f:=T^n+a_{n-1}T^{n-1}\plp a_0\in\mathcal{A}$ has
    factorization pattern $\bfs \lambda$ if and only if there exists
    $\bfs{x}\in\fq^n$ of type $\bfs \lambda$ such that \eqref{eq: fact
        patterns: sym pols and pattern lambda} and
    \begin{equation}\label{eq: fact patterns: systems pattern lambda}
    L_j\big((-1)^{s}\, \Pi_{s}(\bfs Y(\bfs x))\klk -\Pi_1(\bfs
    Y(\bfs x))\big)+\alpha_j=0 \quad(1\le j\le m)
    \end{equation}
    hold. In particular, if $f:=G(\bfs x,T)\in\mathcal{A}_{\bfs\lambda}$
    is square--free, then there are $w(\bfs \lambda)$ elements $\bfs x$
    for which (\ref{eq: fact patterns: systems pattern lambda}) holds.
\end{corollary}
%
%----------------------------------------------------------------------
%----------------------------------------------------------------------
%
\subsection{The number of polynomials in $\mathcal{A}_{\lambda}$}
%
%Let $m$, $n$ and $s$ be positive integers with $q>n$ and  $m \leq s
%\leq n-m-2$.
Given a factorization pattern $\bfs{\lambda}:=1^{\lambda_1} \cdots
n^{\lambda_n}$, consider the set $\mathcal{A}_{\bfs\lambda}$ of
elements of the family $\mathcal{A}\subset \mathcal{P}$ of (\ref{eq:
intro: definition A}) having factorization pattern $\bfs{\lambda}$.
In this section we estimate the number of elements of
$\mathcal{A}_{\bfs\lambda}$. For this purpose, %we shall express this
%quantity in terms of the number of $\fq$--rational zeros of certain
%symmetric polynomials of $\fq[\bfs X]:=\fq[X_1,\ldots, X_n]$.
%
in Corollary \ref{coro: fact patterns: systems pattern lambda} we
associate to $\mathcal{A}_{\bfs \lambda}$ the following polynomials
of $\fq[\bfs X]$:
\begin{equation}\label{eq: geometry: def R_j}
R_j:=R_j^{\bfs\lambda}:= L_j\big((-1)^{s}\, \Pi_{s}(\bfs Y(\bfs
X))\klk -\Pi_1(\bfs Y(\bfs X))\big)+\alpha_j\quad (1\le j\le m).
\end{equation}
Up to the linear change of coordinates defined by $\bfs Y:=(Y_1\klk
Y_n)$, where $\bfs Y$ is the vector of linear forms of $\cfq[\bfs
X]$ defined  in \eqref{eq: fact patterns: def linear forms Y}, we
may express each $R_j$ as a linear polynomial in the first $s$
elementary symmetric polynomials $\Pi_1,\ldots,\Pi_{s}$ of $\fq[\bfs
Y]$. More precisely, let $Z_1\klk Z_{s}$ be new indeterminates over
$\cfq$. Then we may write
$$%\begin{equation}\label{polynomials  Ri patterns}
R_j=S_j(\Pi_1,\ldots,\Pi_{s})\quad (1\le j\le m),
$$%\end{equation}
where $S_1\klk S_m\in \fq[Z_1,\ldots,Z_{s}]$ are defined as
$S_j:=L_j((-1)^sZ_{s},\ldots,-Z_1)+\alpha_j$ ($1 \leq j \leq m$).

We observe that $S_1,\ldots,S_m$ are elements of degree $1$ whose
homogeneous components of degree 1 are linearly independent in
$\cfq[Z_1,\ldots,Z_{s}]$. It follows that the Jacobian matrix
$(\partial \bfs{S}/\partial \bfs{Z})(\boldsymbol{z})$ of
$\boldsymbol{S}:=(S_1\klk S_m)$ with respect to $\bfs{Z}
:=(Z_1,\ldots,Z_{s})$ has full rank $m$ for every
$\boldsymbol{z}\in\A^s$. Furthermore, it is easy to see that
$S_1\ldots,S_m$ form a regular sequence. In other words,
$S_1,\ldots,S_m$ satisfy hypotheses ($\sf H_1$) and ($\sf H_2$) of
Section \ref{section: geometry symm complete inters}.

On the other hand, by assumption the Jacobian matrix $(\partial \bfs
S/\partial \bfs Z)$ is lower triangular in row--echelon form, $1
\leq i_1<\ldots <i_m\leq s$ denoting the positions corresponding to
the pivots. This shows that the components
$S_1^{\mathrm{wt}}=c_1Z_{i_1},\ldots,S_m^{\mathrm{wt}}=c_mZ_{i_1}$
of highest weight of $S_1,\ldots,S_m$ are linearly independent
homogeneous polynomials of degree $1$. Hence, $S_1,\ldots,S_m$
satisfy hypothesis ($\sf H_3$). Finally, as the integers $m$, $n$
and $s$ satisfy the inequalities $m \leq s\leq n-m-2$, we are able
to apply Theorem \ref{theorem: nb point V_r}.

Taking into account that $\mathrm{deg}(R_j)=i_j$ $(1 \leq j\leq m)$,
by the previous considerations and Corollary \ref{coro: nb point V_r
distinct coordinates} we obtain the following result.
\begin{theorem}\label{Ineq: inequality of V(fq) I}
Let $s$, $m$ and $n$ be positive integers with $m \leq s \leq
n-m-2$, and let $R_1,\ldots,R_m \in \fq[X_1,\ldots,X_n]$ be the
polynomials defined in \eqref{eq: geometry: def R_j}. Then there
exist polynomials $S_1\ldots,S_m \in \cfq[Z_1\ldots,Z_s]$ satisfying
hypothesis $(\sf H_1)$, $(\sf H_2)$ and $(\sf H_3)$ for which
$R_i:=S_i(\Pi_1,\ldots,\Pi_s)$ holds $(1 \leq i \leq m)$.
Furthermore, let $V:=V(R_1,\ldots,R_m)\subset \A^n$, let
$$V^{=}:=\mathop{\bigcup_{1\le i\le n}}_{ 1\leq j_1 <j_2\leq
    \lambda _i,\, \,  1\leq k_1 <k_2 \leq i}
V\cap\{Y_{\ell_{i,j_1}+k_1}=Y_{\ell_{i,j_2}+k_2}\},
$$
and $V^{\neq}:=V\setminus V^{=}$, where $Y_{\ell_{i,j}+k}$ are the
linear forms of \eqref{eq: fact patterns: def linear forms Y}. Then
\begin{align*}
%\big||V(\fq)|-q^{n-m}\big|&\leq 14 D_{\bfs L}^3 \delta_{\bfs
%L}^2(q+1)q^{n-m-2},\\
\big||V^{\neq}(\fq)|-q^{n-m}\big|&\leq 14 D_{\bfs L}^3 \delta_{\bfs
L}^2(q+1)q^{n-m-2}+n^2\delta_{\bfs L}q^{n-m-1},
\end{align*}
where $D_{\bfs L}:=\sum_{j=1}^m (i_j-1)$ and $\delta_{\bfs
L}:=i_1\cdots i_m$.
\end{theorem}

%Let $m$, $n$ and $s$ be positive integers with $q>n$ and $m\le s\le
%n-m-2$. Let $A_{n-s},\ldots,A_{n-1}$ be indeterminates over $\cfq$
%and set $\bfs A:=(A_{n-1},\ldots,A_{n-s})$. Let be given sparse
%linear forms $L_1,\ldots,L_m\in\fq[\bfs A]$ which are linearly
%independent and $\bfs\alpha:=(\alpha_1\klk\alpha_m)\in\fq^m$. Set
%$\bfs{L}:=(L_1,\ldots,L_m)$ and let
%$\mathcal{A}:=\mathcal{A}(\bfs{L},\bfs \alpha)$ be the set defined
%in the following way:
%%
%$$\mathcal{A}:=\left\{T^n+a_{n-1}T^{n-1}\plp a_0\in\fq[T]:
%L_j(a_{n-s}\klk a_{n-1})+\alpha_j=0\ (1\le j\le m)\right\}.$$
%%
%As before, we assume that the Jacobian matrix $(\partial \bfs
%L/\partial\bfs A)$ is lower triangular in row echelon form and
%denote by $1\le i_1<\cdots<i_m\le s$ the positions corresponding to
%the pivots. Given a factorization pattern
%$\bfs\lambda:=1^{\lambda_1}\cdots n^{\lambda_n}$, our aim is to
%determine the asymptotic behavior of the number
%$|\mathcal{A}_{\bfs{\lambda}}|$ of elements of $\mathcal{A}$ with
%factorization pattern $\bfs\lambda$.
%
Corollary \ref{coro: fact patterns: systems pattern lambda} relates
the number $|V(\fq)|$ of common $\fq$--rational zeros of $R_1\klk
R_m$ to the quantity $|\mathcal{A}_{\bfs{\lambda}}|$. More
precisely, let $\bfs x:=(\bfs x_{i,j}:1\le i\le n,1\le j\le
\lambda_i)\in\fq^n$ be an $\fq$--rational zero of $R_1\klk R_m$ of
type $\bfs\lambda$. Then $\bfs x$ is associated with an element
$f\in\mathcal{A}_{\bfs\lambda}$ having $Y_{\ell_{i,j}+k}(\bfs
x_{i,j})$ as an $\fqi$--root for $1\le i\le n$, $1\le j\le\lambda_i$
and $1\le k\le i$, where, $Y_{\ell_{i,j}+k}$ is the linear form
defined as in \eqref{eq: fact patterns: def linear forms Y}.

Let $\mathcal{A}_{\bfs\lambda}^{sq}:=\{f\in
\mathcal{A}_{\bfs\lambda}: f \mbox{ is square--free}\}$ and
$\mathcal{A}_{\bfs\lambda}^{nsq}:=\mathcal{A}_{\bfs
\lambda}\setminus \mathcal{A}_{\bfs\lambda}^{sq}$. Corollary
\ref{coro: fact patterns: systems pattern lambda} further asserts
that any element $f\in \mathcal{A}_{\bfs\lambda}^{sq}$ is associated
with $w(\bfs\lambda):=\prod_{i=1}^n i^{\lambda_i}\lambda_i!$ common
$\fq$--rational zeros of $R_1\klk R_m$ of type $\bfs\lambda$.
Observe that $\bfs x\in\fq^n$ is of type $\bfs\lambda$ if and only
if $Y_{\ell_{i,j}+k_1}(\bfs x) \neq Y_{\ell_{i,j}+k_2}(\bfs x)$ for
$1\leq i\leq n$, $1\leq j\leq \lambda_i$ and $1\leq k_1 <k_2 \leq
i$. Furthermore, an $\bfs x\in\fq^n$ of type $\bfs\lambda$ is
associated with $f\in\mathcal{A}_{\bfs\lambda}^{sq}$ if and only if
$Y_{\ell_{i,j_1}+k_1}(\bfs x) \neq Y_{\ell_{i,j_2}+k_2}(\bfs x)$ for
$1\leq i\leq n$, $1\leq j_1<j_2\leq \lambda_i$ and $1\leq k_1 <k_2
\leq i$. As a consequence, we see that
$|\mathcal{A}_{\bfs\lambda}^{sq}| =\mathcal{T}(\bfs\lambda)
\big|V^{\neq}(\fq)\big|$, which implies
$$
\big||\mathcal{A}_{\bfs\lambda}^{sq}|
-\mathcal{T}(\bfs\lambda)\,q^{n-m}\big| =
\mathcal{T}(\bfs\lambda)\,\big||V^{\neq}(\fq)|-q^{n-m}\big|.
$$
%
%Next we bound $|V^{=}(\fq)|$. Theorem \ref{theorem: proj closure of
%Vr is abs irred} shows that $V$ is a complete intersection which is
%regular in codimension $2$. Therefore, by Theorem \ref{theorem:
%normal complete int implies irred} we conclude that $V$ is
%absolutely irreducible. Hence we have that
%$V\cap\{Y_{\ell_{i,j_1}+k_1}=Y_{\ell_{i,j_2}+k_2}\}$ has dimension
%at most $n-m-1$ for every $1\le i\le n$, $1\leq j_1 <j_2\leq \lambda
%_i$ and $1\leq k_1 <k_2 \leq i$. We conclude that $V^=$ has
%dimension at most $n-m-1$. On the other hand, by
% the B\'ezout inequality we have
%%
%$$\deg V^=\le \deg V\sum_{i=1}^ni^2\lambda_i^2\le n^2\delta_{\bfs L}.$$
%%
%Then \eqref{eq: upper bound -- affine gral} implies
%%
%\begin{equation}\label{eq: estimates: upper bound V^=(fq)}
%|V^{=}(\fq)|\le \deg V^=\,q^{n-m-1}\le n^2\delta_{\bfs L} q^{n-m-1}.
%\end{equation}
%%
From Theorem \ref{Ineq: inequality of V(fq) I} we deduce that
\begin{align*}
\big||\mathcal{A}_{\bfs\lambda}^{sq}|
-\mathcal{T}(\bfs\lambda)\,q^{n-m}\big| &\le
\,\mathcal{T}(\bfs\lambda)\big(14\, D_{\bfs
        L}^3\delta_{\bfs L}^2(q+1)q^{n-m-2}+ n^2\delta_{\bfs
        L}q^{n-m-1}\big)\\
&\le q^{n-m-1} \mathcal{T}(\bfs\lambda)\,\big(21\, D_{\bfs
    L}^3\delta_{\bfs L}^2+ n^2\delta_{\bfs L}\big).
\end{align*}
%By Remark \ref{Ineq: inequality of V(fq) I} we conclude that
%
%\begin{equation*}
%\big||\mathcal{A}_{\bfs\lambda}^{sq}|
%-\mathcal{T}(\bfs\lambda)\,q^{n-m}\big|\le q^{n-m-1}
%\mathcal{T}(\bfs\lambda)\,\big(21\, D_{\bfs
%   L}^3\delta_{\bfs L}^2+ n^2\delta_{\bfs L}\big).
%\end{equation*}

Now we are able to estimate $|\mathcal{A}_{\bfs \lambda}|$. We have
\begin{align}
\big||\mathcal{A}_{\bfs\lambda}|
-\mathcal{T}(\bfs\lambda)\,q^{n-m}\big|&=
\big||\mathcal{A}_{\bfs\lambda}^{sq}|+
|\mathcal{A}_{\bfs\lambda}^{nsq}|-\mathcal{T}(\bfs\lambda)q^{n-m}\big|\nonumber\\
&\le q^{n-m-1} \mathcal{T}(\bfs\lambda)\,\big(21\, D_{\bfs
    L}^3\delta_{\bfs L}^2+ n^2\delta_{\bfs L}\big)+
|\mathcal{A}_{\bfs\lambda}^{nsq}|. \label{eq: estimates: estimate
A_lambda aux}
\end{align}
%
%Furthermore,
%%
%$$
%\big||\mathcal{A}_{\bfs\lambda}^{nsq}|
%-\mathcal{T}(\bfs\lambda)|V^{=}(\fq)|\big| =|\mathcal{A}_{\bfs
%    \lambda}^{nsq}|-\mathcal{T}(\bfs\lambda)\, |V^{=}(\fq)| \leq
%|\mathcal{A}_{\bfs \lambda}^{nsq}|.
%$$
%
It remains to obtain an upper bound for
$|\mathcal{A}_{\bfs\lambda}^{nsq}|$. To this end, we observe that
$f\in \mathcal{A}$ is not square--free if and only if its
discriminant is equal to zero. Let $\mathcal{A}^{nsq}$ be {\em
discriminant locus} of $\mathcal{A}$, i.e., the set of elements of
$\mathcal{A}$ whose discriminant is equal to zero. In \cite{FrSm84}
and \cite{MaPePr14} discriminant loci are studied. In particular,
from \cite{FrSm84} one easily deduces that the discriminant locus
$\mathcal{A}^{nsq}$ is the set of $\fq$--rational points of a
hypersurface of degree at most $n(n-1)$ of a suitable
$(n-m)$--dimensional affine space. Then \eqref{eq: upper bound --
affine gral} implies
\begin{equation}\label{eq: estimates: upper bound discr locus}
|\mathcal{A}_{\bfs\lambda}^{nsq}|\le |\mathcal{A}^{nsq}|\le n(n-1)
\,q^{n-m-1}.
\end{equation}
Hence, combining \eqref{eq: estimates: estimate A_lambda aux} and
\eqref{eq: estimates: upper bound discr locus} we conclude that
\begin{align*}
\big||\mathcal{A}_{\bfs\lambda}|
-\mathcal{T}(\bfs\lambda)\,q^{n-m}\big|&\le
\mathcal{T}(\bfs\lambda)\,\big||V(\fq)|-q^{n-m}\big|+ n^2 q^{n-m-1}\\
&\le q^{n-m-1}\big(21\,\mathcal{T}(\bfs\lambda)\, D_{\bfs
    L}^3\delta_{\bfs L}^2+ \mathcal{T}(\bfs\lambda)\,n^2\delta_{\bfs L}+
    n^2\big).
\end{align*}
%Hence, by Proposition \ref{Ineq: inequality of V(fq) I} we conclude that
%\begin{equation*}
%\big||\mathcal{A}_{\bfs\lambda}|
%-\mathcal{T}(\bfs\lambda)\,q^{n-m}\big|\le
%q^{n-m-1}\big(21\,\mathcal{T}(\bfs\lambda)\, D_{\bfs
%   L}^3\delta_{\bfs L}^2+ n^2\big).
%\end{equation*}
%
In other words, we have the following result.
\begin{theorem}\label{theorem: estimate fact patterns}
    For $q>n$ and $m\le s \le n-m-2$, we have that
    \begin{align*}
    \big||\mathcal{A}_{\bfs\lambda}^{sq}|
    -\mathcal{T}(\bfs\lambda)\,q^{n-m}\big|&\le q^{n-m-1}
    \mathcal{T}(\bfs\lambda)\,\big(21\, D_{\bfs
        L}^3\delta_{\bfs L}^2+ n^2\delta_{\bfs L}\big),\\
    \big||\mathcal{A}_{\bfs\lambda}|
    -\mathcal{T}(\bfs\lambda)\,q^{n-m}\big|&\le
    q^{n-m-1}\big(21\,\mathcal{T}(\bfs\lambda)\, D_{\bfs
        L}^3\delta_{\bfs L}^2+ \mathcal{T}(\bfs\lambda)\,n^2\delta_{\bfs L}+n^2\big),
    \end{align*}
    where $ \delta_{\bfs L}:=i_1\cdots i_m$ and $D_{\bfs
        L}:=\sum_{j=1}^m (i_j-1)$.
\end{theorem}

This result strengthens \eqref{eq: intro: Cohen72} in several
aspects. The first one is that the hypotheses on the linear families
$\mathcal{A}$ in the statement of Theorem \ref{theorem: estimate
fact patterns} are easy to verify. The second aspect is that our
error term is of order $\mathcal{O}(q^{n-m-1})$, and we provide
explicit expressions for the constants underlying the
$\mathcal{O}$--notation with a good behavior. Furthermore, our
result is valid without any restriction on the characteristic $p$ of
$\fq$. On the other hand, Theorem \ref{theorem: estimate fact
patterns} holds for $m \leq n/2 -1$, while \eqref{eq: intro:
Cohen72} holds for $m$ varying in a much larger range of values.

A classical case which has received particular attention is that of
the elements of $\mathcal{P}$ having certain coefficients
prescribed. Therefore, we briefly state what we obtain in this case.
For this purpose, given $0< i_1<i_2<\cdots< i_m\le n$ and $\bfs
\alpha:=(\alpha_{i_1}\klk \alpha_{i_m})\in\fq^m$, set
$\mathcal{I}:=\{i_1\klk i_m\}$ and consider the set
$\mathcal{A}^\mathcal{I}:= \mathcal{A}(\mathcal{I},\bfs \alpha)$
defined in the following way:
\begin{equation}\label{eq: estimates: family A^m}
\mathcal{A}^\mathcal{I}:=\left\{T^n+a_1T^{n-1}\plp a_n\in\fq[T]:
a_{i_j}=\alpha_{i_j}\ (1\le j\le m)\right\}.
\end{equation}
Denote by $\mathcal{A}^{\mathcal{I},sq}$ the set of
$f\in\mathcal{A}^\mathcal{I}$ which are square--free.

For a given factorization pattern $\bfs\lambda$, let $G\in\fq[\bfs
X,T]$ be the polynomial of \eqref{eq: fact patterns: pol G}.
According to Lemma \ref{lemma: fact patterns: sym pols and pattern
lambda}, an element $f\in \mathcal{A}^\mathcal{I}$ has factorization
pattern $\bfs\lambda$ if and only if there exists $\bfs x$ of type
$\bfs\lambda$ such that
$$
(-1)^{i_j}\Pi_{i_j}(\bfs Y(\bfs x))=\alpha_{i_j}\quad(1\le j\le m).
$$
Let $\delta_{\mathcal{I}}:=i_1\cdots i_m$ and
$D_{\mathcal{I}}:=\sum_{j=1}^m(i_j-1)$. From Theorem \ref{theorem:
estimate fact patterns} we deduce the following result.
\begin{corollary}\label{coro: estimates: estimates for presc coeff}
     For $q > n$ and $i_m\le n-m-2$, then
    \begin{align*}
    \big||\mathcal{A}^{\mathcal{I},sq}_{\bfs
        \lambda}|-\mathcal{T}(\bfs\lambda)\,q^{n-m}\big|&\le
    q^{n-m-1}\mathcal{T}(\bfs\lambda)\,\big(21\,D_{\mathcal{I}}^3\,
    \delta_{\mathcal{I}}^2 + n^2\delta_{\mathcal{I}}\big),\\
    \big||\mathcal{A}^\mathcal{I}_{\bfs
        \lambda}|-\mathcal{T}(\bfs\lambda)\,q^{n-m}\big|&\le
    q^{n-m-1}\big(21\,\mathcal{T}(\bfs\lambda)\,D_{\mathcal{I}}^3\,
    \delta_{\mathcal{I}}^2 +\mathcal{T}(\bfs\lambda)\,n^2\delta_{\mathcal{I}}+ n^2\big).
    \end{align*}
\end{corollary}

Observe that, for the sake of simplicity, the enumeration of
coefficients of the elements in the family $\mathcal{A}^\mathcal{I}$
of \eqref{eq: estimates: family A^m} is changed with respect to that
of the family $\mathcal{A}(\bfs L,\bfs \alpha)$ of \eqref{eq: intro:
definition A}. Therefore, the conditions $m\le s\le n-m-2$ in
Theorem \ref{theorem: estimate fact patterns} are now expressed as
$i_m\le n-m-2$.
%
%----------------------------------------------------------------------
%----------------------------------------------------------------------
%----------------------------------------------------------------------
%----------------------------------------------------------------------
%----------------------------------------------------------------------
%----------------------------------------------------------------------
%----------------------------------------------------------------------
%----------------------------------------------------------------------
%
\section{The average cardinality of value sets}
\label{section: average card value sets}
Next we consider the problem of estimating the average cardinality
of value sets of families of univariate polynomials with
coefficients in $\fq$. More precisely, in this section we estimate
the average cardinality of the value set of any family of monic
polynomials of $\fq[T]$ of degree $n$ for which $s$ consecutive
coefficients $a_{n-1},\ldots,a_{n-s}$ are fixed.

As we said before, our approach reduces the question to estimate the
number of $\fq$--rational points with pairwise--distinct coordinates
of a certain family of complete intersections defined over $\fq$. As
the polynomials defining such complete intersections are symmetric,
we shall be able to apply Corollary \ref{coro: nb point V_r distinct
coordinates} to obtain a suitable estimate on the number of
$\fq$--rational points of these complete intersections.

We begin by recalling the basic notions and the previous results
related to the study of the cardinality of value sets of univariate
polynomials defined over a finite field.

Let $T$ an indeterminate over $\fq$ and let $f\in \fq[T]$. We denote
by $\mathcal{V}(f)$ the cardinality of the value set of $f$, namely
$\mathcal{V}(f):=|\{f(c):c\in\fq\}|$ (cf. \cite{LiNi83}). The
quantity $\mathcal{V}(f)$ has been extensively studied for arbitrary
polynomials. Exact formulas for $\mathcal{V}(f)$ are established for
certain particular classes of polynomials (for example, polynomials
of low degree).

Concerning the behavior of $\mathcal{V}(f)$ for ``large'' sets of
elements of $\fq[T]$, Birch and Swinnerton--Dyer established the
following significant result \cite{BiSD59}: for fixed $n\ge 1$, if
$f$ is a generic polynomial of degree $n$, then
$$\mathcal{V}(f)=\mu_n\,q+
\mathcal{O}(q^{1/2}),$$
where $\mu_n:=\sum_{r=1}^n{(-1)^{r-1}}/{r!}$ and the constant
underlying the $\mathcal{O}$--notation depends only on $n$. Results
on the average value $\mathcal{V}(n,0)$ of $\mathcal{V}(f)$ when $f$
ranges over all monic polynomials in $\fq[T]$ of degree $n$ with
$f(0)=0$ were obtained by Uchiyama \cite{Uchiyama55b} and improved
by Cohen \cite{Cohen73}. More precisely, in \cite[\S 2]{Cohen73} it
is shown that
$$\mathcal{V}(n,0)=\sum_{r=1}^n(-1)^{r-1}\binom{q}{r}q^{1-r}=\mu_n\,q
+\mathcal{O}(1).$$
A variant of this problem, considered by Uchiyama and Cohen, asks
for results on the average cardinality of the value set of the set
of polynomials $f\in\fq[T]$ of given degree were some coefficients
are fixed. For this problem, in \cite{Uchiyama55b} and
\cite{Cohen72} the authors obtain the following result. Consider the
family of monic polynomials of $\fq[T]$ of degree $n$, were $s$
consecutive coefficients are fixed, with $1\le s\le n-2$. More
precisely, denote by $\bfs{a}:=(a_{n-1}\klk a_{n-s})\in\fq^s$ the
vector of values of the coefficients to be fixed and consider the
set of $f_{\boldsymbol{b}}$ of the form
$$f_{\boldsymbol{b}}:=f_{\boldsymbol{b}}^{\bfs{a}}
:=T^n+\sum_{i=1}^sa_{n-i}T^{n-i}+\sum_{i=s+1}^nb_{n-i}T^{n-i}$$
for every $\boldsymbol{b}:=(b_{n-s-1}\klk b_0)$. Then for
$p:=\mathrm{char}(\fq)>n$,
\begin{equation}\label{eq: average value set}
\mathcal{V}(n,s,\bfs{a}):=
\frac{1}{q^{n-s}}\sum_{\boldsymbol{b}\in\fq^{n-s}}\mathcal{V}(f_{\boldsymbol{b}})=
\mu_n\,q+\mathcal{O}(q^{1/2}),
\end{equation}
where the constant underlying the $\mathcal{O}$--notation depends
only on $n$ and $s$. In this section we obtain an strengthened
explicit version of (\ref{eq: average value set}), which holds
without any restriction on the characteristic $p$ of $\fq$.
%
%----------------------------------------------------------------------
%----------------------------------------------------------------------
%
\subsection{Value sets in terms of interpolating sets}
Given $\bfs{a}:=(a_{n-1}\klk a_{n-s})\in\fq^s$, let
$$f_{\bfs{a}}:=T^n+a_{n-1}T^{n-1}\plp a_{n-s}T^{n-s}.$$
%
%For each $\bfs{b}\in \fq^{n-s}$ we denote by
%$f_{\bfs{b}}:=f_{\bfs{b}}^{\bfs{a}}\in\fq[T]$ the following
%polynomial
%%
%$$f_{\bfs{b}}:=f_{\bfs{a}}+
%b_{n-s-1}T^{n-s-1}\plp b_1T+b_0.$$ Recall that
%$\mathcal{P}:=\mathcal{P}_n$ is the set of monic polynomials of
%degree $n$.
Then we consider the set $\mathcal{A}:=\mathcal{A}(n,s,{\bfs{a}})$
defined in the following way:
$$\mathcal{A}:=\{f_{\bfs b}:=f_{\bfs{a}}+b_{n-s-1}T^{n-s-1}\plp b_0:\,
\bfs{b}:=(b_{n-s-1}\klk b_0)\in \fq^{n-s} \}.$$
%namely, the set of
%monic polynomials of degree $n$ with $s$ coefficients fixed.
Hence $\mathcal{V}(n,s,\bfs{a})$ is the average value of
$\mathcal{V}(f)$ when $f$ ranges over elements of $\mathcal{A}$,
that is,
$$%\begin{equation}\label{eq: value set in family}
\mathcal{V}(n,s,\bfs{a}):= \frac{1}{|\mathcal{A}|}\sum_{f\in
\mathcal{A}}\mathcal{V}(f)=\frac{1}{q^{n-s}}\sum_{\boldsymbol{b}\in\fq^{n-s}}\mathcal{V}(f_{\boldsymbol{b}}).
$$%\end{equation}
Following \cite{CeMaPePr14}, in order to estimate
$\mathcal{V}(n,s,\bfs{a})$ we express this quantity in terms of the
number $\chi(\bfs{a},r)$ of certain ``interpolating sets'' with
$n-s+1\le r\le n$.

Observe that for any $\bfs{b}:=(b_{n-s-1},\ldots,b_0)\in\fq^{n-s}$,
the cardinality $\mathcal{V}(f_{\bfs{b}})$ of the value set of
$f_{\bfs{b}}$ equals the number of elements $\beta_0\in\fq$ for
which $f_{\bfs{b}}+\beta_0$ has at least one root in $\fq$. Denote
by $\mathcal{P}:=\mathcal{P}_n$ the set of monic polynomials of
$\fq[T]$ of degree $n$, let $\mathcal{N}:\mathcal{P}\to\Z_{\ge 0}$
be the random variable which counts the number of roots in $\fq$ and
$\bfs{1}_{\{\mathcal{N}>0\}}:\mathcal{P}\to\{0,1\}$ the
characteristic function of the set of elements of $\mathcal{P}$
having at least one root in $\fq$. From the assertions above we
deduce the following identity:
$$\sum_{\bfs{b}\in\fq^{n-s}}\mathcal{V}(f_{\bfs{b}})=
\sum_{\beta_0\in\fq}\sum_{\bfs{b}\in\fq^{n-s}}
\bfs{1}_{\{\mathcal{N}>0\}}
(f_{\bfs{b}}+\beta_0)=q\cdot\big|\{g\in\fq[T]_{n-s-1}:
\mathcal{N}(f_{\bfs{a}}+g)>0 \}\big|,$$
where $\fq[T]_{n-s-1}$ is the set of elements of $\fq[T]$ of degree
at most $n-s-1$. For $\mathcal{X}\subseteq\fq$, we define
$\mathcal{S}_{\mathcal{X}}^{\bfs{a}}$ as the set of
$g\in\fq[T]_{n-s-1}$ which interpolate $-f_{\bfs{a}}$ at all the
points of $\mathcal{X}$, namely
$$
\mathcal{S}_{\mathcal{X}}^{\bfs{a}}:=\{g\in\fq[T]_{n-s-1}:
(f_{\bfs{a}}+g)(x)=0\textrm{ for any }x\in\mathcal{X}\}.
$$
For $r\in\N$ we shall use the symbol $\mathcal{X}_r$ to denote a
subset of $\fq$ of $r$ elements.

With this terminology, we have the following combinatorial
expression of $\mathcal{V}(n,s,\bfs{a})$, whose proof can be found
in \cite[Theorem 2.1]{CeMaPePr14}.
\begin{proposition}\label{prop: reduction value sets to interp sets}
Given $s,n\in\N$ with $1\le s\le n-1$, we have
\begin{equation}\label{eq: our formula for value sets}
\mathcal{V}(n,s,\bfs{a})= \sum_{r=1}^{n-s}(-1)^{r-1}\binom {q} {r}
q^{1-r}+\frac{1}{q^{n-s-1}}\sum_{r=n-s+1}^{n} (-1)^{r-1}
\chi(\bfs{a},r),\end{equation}
where  $\chi(\bfs{a},r)$ is the number of subsets $\mathcal{X}_r$ of
$\fq$ of $r$ elements such that there exists $g\in\fq[T]_{n-s-1}$
with $(f_{\bfs{a}}+g)|_{\mathcal{X}_r}\equiv 0$.
\end{proposition}
According to this result, the asymptotic behavior of
$\mathcal{V}(n,s,\bfs{a})$ is determined by that of
$\chi(\bfs{a},r)$ for $n-s+1\leq r \leq n$. We shall show that each
$\chi(\bfs{a},r)$ can be expressed as the number of $\fq$--rational
points with pairwise--distinct coordinates of an affine
$\fq$--variety defined by symmetric polynomials.
%
%----------------------------------------------------------------------
%----------------------------------------------------------------------
%
\subsection{The number $\chi(\bfs{a},r)$ in terms of zeros of symmetric polynomials}
\label{subsec: relating interpol to rat points} Fix
$\bfs{a}\in\fq^{n-s}$ and $r$ with $n-s+1\leq r \leq n$. To estimate
$\chi(\bfs{a},r)$, we follow the approach of \cite{CeMaPePr14}.

Fix a set $\mathcal{X}_r:=\{x_1\klk x_r\}\subset\fq$ of $r$ elements
and $g\in\fq[T]_{n-s-1}$. Then $g$ belongs to
$\mathcal{S}_{\mathcal{X}_r}^{\bfs{a}}$ if and only if
$(T-x_1)\cdots(T-x_r)$ divides $f_{\bfs{a}}+g$ in $\fq[T]$. Since
$\deg g\le n-s-1<r$, we deduce that $-g$ is the remainder of the
division of $f_{\bfs{a}}$ by $(T-x_1)\cdots(T-x_r)$. In other words,
the set $\mathcal{S}_{\mathcal{X}_r}^{\bfs{a}}$ is not empty if and
only if the remainder of the division of $f_{\bfs{a}}$ by
$(T-x_1)\cdots (T-x_r)$ has degree at most $n-s-1$.

Let $X_1,\ldots, X_r$ be indeterminates over $\cfq$, let
$\bfs{X}:=(X_1\klk X_r)$ and
$$Q=(T-X_1)\cdots(T-X_r)\in\fq[\bfs{X}][T].$$
There exists $R_{\bfs{a}}\in\fq[\bfs{X}][T]$ with $\deg
R_{\bfs{a}}\leq r-1$ such that the following relation holds:
$$%\begin{equation}\label{eq: congruence mod Q}
f_{\bfs{a}}\equiv R_{\bfs{a}}\mod{Q}.
$$%\end{equation}
Write $R_{\bfs{a}}=R_{r-1}^{\bfs{a}}(\bfs{X})T^{r-1}\plp
R_0^{\bfs{a}}(\bfs{X})$. Then $R_{\bfs{a}}(x_1\klk x_r,T)\in\fq[T]$
is the remainder of the division of $f_{\bfs{a}}$ by $(T-x_1)\cdots
(T-x_r)$. As a consequence, the set
$\mathcal{S}_{\mathcal{X}_r}^{\bfs{a}}$ is not empty if and only if
the following identities hold:
\begin{equation}\label{eq: system defining Vr}
R_j^{\bfs{a}}(x_1\klk x_r)=0\quad (n-s\le j\le r-1).
\end{equation}
On the other hand, if there exists $\bfs x:=(x_1\klk x_r)\in\fq^r$
with pairwise--distinct coordinates such that (\ref{eq: system
defining Vr}) holds, then the remainder of the division of
$f_{\bfs{a}}$ by $Q(\bfs{x},T)=(T-x_1)\cdots (T-x_r)$ is a
polynomial $r_{\bfs{a}}:=R_{\bfs{a}}(\bfs{x},T)$ of degree at most
$n-s-1$. This shows that $\mathcal{S}_{\mathcal{X}_r}^{\bfs{a}}$ is
not empty, where $\mathcal{X}_r:=\{x_1\klk x_r\}$. In other words,
we have the following result.
\begin{lemma}\label{lemma: system defining Vr}
Let $s,n\in\N$ with $1\le s\le n-2$, let $R_j^{\bfs{a}}$ $(n-s\le
j\le r-1)$ be the polynomials of (\ref{eq: system defining Vr}) and
let $\mathcal{X}_r:=\{x_1\klk x_r\}\subset\fq$ be a set with $r$
elements. Then $\mathcal{S}_{\mathcal{X}_r}^{\bfs{a}}$ is not empty
if and only if (\ref{eq: system defining Vr}) holds.
\end{lemma}

It follows that the number $\chi(\bfs{a},r)$ of sets
$\mathcal{X}_r\subset \fq$ of $r$ elements such that
$\mathcal{S}_{\mathcal{X}_r}^{\bfs{a}}$ is not empty equals the
number of points $\bfs{x}:=(x_1,\ldots, x_r)\in\fq^r$ with
pairwise--distinct coordinates satisfying (\ref{eq: system defining
Vr}), up to permutations of coordinates, namely $1/r!$ times the
number of solutions $\bfs{x}\in\fq^r$ of the following system of
equalities and non-equalities:
$$%\begin{equation}%\label{eq: system rational solutions Vr}
R_j^{\bfs{a}}(X_1,\ldots, X_r)=0\ \ (n-s\le j\le r-1),\ \prod_{1\le
i<j\le r}(X_i-X_j)\not=0.
$$%\end{equation}

Fix $r$ with $n-s+1\le r\le n$ and assume that $2(s+1)\le n$ holds.
%
%In Section \ref{subsec: relating interpol to rat points} we obtain
%polynomials $R_j^{\bfs{a}}\in \fq[X_1,\ldots, X_r]$ ($n-s\le j\le
%r-1$) with the following property: for a given set
%$\mathcal{X}_r:=\{x_1\klk x_r\}\subset\fq$ of $r$ elements, the set
%$\mathcal{S}_{\mathcal{X}_r}^{\bfs{a}}$ is not empty if and only if
%$(x_1\klk x_r)$ is a common zero of $R_{n-s}^{\bfs{a}}\klk
%R_{r-1}^{\bfs{a}}$.
%
Next we show how the polynomials $R_j^{\bfs{a}}$ can be expressed in
terms of the elementary symmetric polynomials $\Pi_1 ,\ldots, \Pi_s$
of $\fq[X_1,\ldots, X_r]$. This is the content of \cite[Lemma
2.3]{CeMaPePr14} and \cite[Proposition 2.4]{CeMaPePr14}, whose
proofs are sketched here to illustrate the way in which the
elementary symmetric polynomials enter into play.

The first step is to obtain an expression for the remainder of the
division of $T^j$ by $Q:=(T-X_1)\cdots (T-X_r)$ for $r\le j\le n$.
For convenience of notation, we denote $\Pi_0:= 1$.
\begin{lemma}\label{lemma: formula Hij}
For $r\le j\le n$, the following congruence relation holds:
\begin{equation}
\label{eq: reduced expression T j} T^j\equiv
H_{r-1,j}T^{r-1}+H_{r-2,j}T^{r-2}+\cdots+H_{0,j}\mod Q,
\end{equation}
where each $H_{i,j}$ is equal to zero or an homogeneous element of
$\fq[X_1\klk X_r]$ of degree $j-i$. Furthermore, for $j-i\le r$, the
polynomial $H_{i,j}$ is a monic element of
$\fq[\Pi_1\klk\Pi_{j-i-1}][\Pi_{j-i}]$, up to a nonzero constant of
$\fq$, of degree $1$ in $\Pi_{j-i}$.
\end{lemma}
\begin{proof}
We argue by induction on $j\ge r$. Taking into account that
\begin{equation}\label{eq: reduced expression T r+1}
T^r\equiv  \Pi_1T^{r-1}-\Pi_2T^{r-2}+\cdots+(-1)^{r-1}\Pi_r
\mod{Q},\end{equation}
we immediately deduce (\ref{eq: reduced expression T j}) for $j=r$.
Next assume that (\ref{eq: reduced expression T j}) holds for a
given $j$ with $r\le j$. Multiplying both sides of (\ref{eq: reduced
expression T j}) by $T$ and combining with (\ref{eq: reduced
expression T r+1}) we obtain:
\begin{align*}
T^{j+1}&\equiv H_{r-1,j}T^r+H_{r-2,j}T^{r-1}+
\cdots+H_{0,j}T\\
&\equiv(\Pi_1H_{r-1,j}+
H_{r-2,j})T^{r-1}+\cdots+((-1)^{r-2}\Pi_{r-1}H_{r-1,j}+H_{0,j})T\\
&\quad +(-1)^{r-1}\Pi_rH_{r-1,j},
\end{align*}
where both congruences are taken modulo $Q$. Define
\begin{align*}
H_{k,j+1}&:=(-1)^{r-1-k}\Pi_{r-k}H_{r-1,j}+H_{k-1,j}\ \textit{for}\
1\le k\le r-1,\\H_{0,j+1}&:=(-1)^{r-1}\Pi_rH_{r-1,j}.
\end{align*} Then we have
$$T^{j+1}\equiv H_{r-1,j+1}T^{r-1}+H_{r-2,j+1}T^{r-2}+\cdots+H_{0,j+1}
\mod{Q}.$$
%
%There remains to prove
It can be seen that the polynomials $H_{k,j+1}$ have the form
asserted (see the proof of \cite[Lemma 2.3]{CeMaPePr14} for
details).
%
%Fix $k$ with $1\le k\le r-1$. Then
%$H_{k,j+1}=(-1)^{r-1-k}\Pi_{r-k}H_{r-1,j}+H_{k-1,j}$. By the
%inductive hypothesis we have that $H_{r-1,j}$ and $H_{k-1,j}$ are
%equal to zero or homogeneous polynomials of degree $j-r+1$ and
%$j-k+1$ respectively. We easily conclude that $H_{k,j+1}$ is equal
%to zero or homogeneous of degree $j-k+1$. Further, for $j+1-k\le r$,
%since $\max\{r-k,j-r+1\}\le j-k<r$ we see that $\Pi_{r-k}H_{r-1,j}$
%is an element of the polynomial ring $\fq[\Pi_1\klk\Pi_{j-k}]$. On
%the other hand, $H_{k-1,j}$ is a monic element of
%$\fq[\Pi_1\klk\Pi_{j-k}][\Pi_{j-k+1}]$, up to a nonzero constant of
%$\fq$, which implies that so is $H_{k,j+1}$.
%
%Finally, for $k=0$ we have $H_{0,j+1}:=(-1)^{r-1}\Pi_rH_{r-1,j}$,
%which shows that $H_{0,j+1}$ is equal to zero or an homogeneous
%polynomials of $\fq[X_1\klk X_r]$ of degree $r+j-r+1=j+1$. This
%finishes the proof of the lemma.
\end{proof}

%We observe that an explicit expression of the polynomials $H_{i,j}$
%can be obtained following the approach of \cite[Proposition
%2.2]{CaMaPr12}. As we do not need such an explicit expression we
%shall not pursue this point any further.

Finally we express each $R_j^{\bfs{a}}$ in terms of the polynomials
$H_{i,j}$.
\begin{proposition}  \label{prop: formula for Rj}
Let $s,n\in\N$ with $1\le s\le n-2$ and $2(s+1)\le n$. For $n-s\le
j\le r-1$, the following identity holds:
\begin{equation} \label{eq: expression for Rj}
R_j^{\bfs{a}}=a_j+\sum_{i=r}^ n a_iH_{j,i},
\end{equation}
where the polynomials $H_{j,i}$ are defined in Lemma \ref{lemma:
formula Hij}. In particular, $R_j^{\bfs{a}}$ is a monic element of
$\fq[\Pi_1\klk \Pi_{n-1-j}][\Pi_{n-j}]$ of degree $n-j\le s$ for
$n-s\le j\le r-1$.
\end{proposition}
\begin{proof}
By Lemma \ref{lemma: formula Hij} we have the following congruence
relation for $r\le j\le n$:
$$T^j\equiv
H_{r-1,j}T^{r-1}+H_{r-2,j}T^{r-2}+\cdots+H_{0,j} \mod{Q}.$$
Hence we obtain
\begin{align*}
\sum_{j=n-s}^n a_jT^j&=\sum_{j=n-s}^{r-1}a_jT^j+\sum_{j=r}^n a_jT^j\\
&\equiv \sum_{j=n-s}^{r-1}a_jT^j+\sum_{j=r}^n a_j
\sum_{i=n-s}^{r-1}H_{i,j}T^i+\mathcal{O}(T^{n-s-1}) \mod{Q}\\
&\equiv \sum_{j=n-s}^{r-1}\bigg(a_j+ \sum_{i=r}^n
a_iH_{j,i}\bigg)T^j+\mathcal{O}(T^{n-s-1}) \mod{Q},
\end{align*}
where $\mathcal{O}(T^{n-s-1})$ represents a sum of terms of
$\fq[X_1\klk X_r][T]$ of degree at most $n-s-1$ in $T$. This shows
that the polynomials $R_j^{\bfs{a}}$ have the form asserted in
(\ref{eq: expression for Rj}). Furthermore, we observe that, for
each $H_{j,i}$ occurring in (\ref{eq: expression for Rj}), we have
$i-j\le s\le n-s-2\le r$. This implies that each $H_{j,i}$ is a
monic element of $\fq[\Pi_1\klk \Pi_{i-j-1}][\Pi_{i-j}]$ of degree
$i-j$. It follows that $R_j^{\bfs{a}}$ is a monic element of
$\fq[\Pi_1\klk \Pi_{n-1-j}][\Pi_{n-j}]$ of degree $n-j$ for $n-s\le
j\le r-1$. This finishes the proof.\end{proof}
%
%----------------------------------------------------------------------
%----------------------------------------------------------------------
%
\subsection{An estimate for $\mathcal{V}(n,s,\bfs{a})$}
Proposition \ref{prop: reduction value sets to interp sets} shows
that the asymptotic behavior of $\mathcal{V}(n,s,\bfs{a})$ is
determined by that of $\chi(\bfs{a},r)$ for $n-s+1\le r\le n$. Fix
$r$ with $n-s+1\le r\le n$. In Section \ref{subsec: relating
interpol to rat points} we associate to ${\bfs{a}}$ certain
polynomials $R_j^{\bfs{a}}\in\fq[X_1\klk X_r]$ $(n-s\le j\le r-1)$
with the property that the number of common $\fq$--rational zeros of
$R_{n-s}^{\bfs{a}}\klk R_{r-1}^{\bfs{a}}$ with pairwise distinct
coordinates equals $r!\chi(\bfs{a},r)$, namely
$$\chi(\bfs{a},r)=\frac{1}{r!}\left|\left\{\bfs{x}\in\fq^r:
R_j^{\bfs{a}}(\bfs{x})=0\, (n-s\le j\le r-1),x_k\not= x_l\, (1\le
k<l\le r)\right\}\right|.$$

According to Proposition \ref{prop: formula for Rj}, each
$R^{\bfs{a}}_j$ can be expressed as a polynomial in the elementary
symmetric polynomials $\Pi_1,\ldots,\Pi_s$ of $\fq[X_1,\ldots,X_r]$.
More precisely, if $Y_1,\ldots, Y_s$ are new indeterminates over
$\cfq$, then we may write
$$R^{\bfs{a}}_j=S^{\bfs{a}}_j(\Pi_1,\ldots,\Pi_{n-j})\quad (n-s\leq j \leq r-1),$$
where each $S^{\bfs{a}}_j\in \fq[Y_1,\ldots,Y_{n-j}]$ is a monic
element of $\fq[Y_1,\ldots,Y_{n-1-j}][Y_{n-j}]$ of degree $1$ in
$Y_{n-j}$, up to a nonzero constant of $\fq$. In particular, it is
easy to that
\begin{equation}\label{eq: isomorfismo}
\cfq[Y_1\klk Y_s]/(S_{n-s}^{\bfs{a}}\klk S_j^{\bfs{a}})\simeq\cfq[Y_1\klk Y_{n-j-1}]
\end{equation}
for $n-s\leq j \leq r-1$. Hence, $S_{n-s}^{\bfs{a}}\klk
S_{r-1}^{\bfs{a}}$ form a regular sequence of $\fq[Y_1,\ldots,Y_s]$,
namely these polynomials satisfy hypothesis $(\mathsf{H}_1)$ of
Section \ref{section: geometry symm complete inters}.

Further, by the isomorphism (\ref{eq: isomorfismo}) for $j=r-1$ we
deduce that $S_{n-s}^{\bfs{a}}\klk S_{r-1}^{\bfs{a}}$ form a radical
ideal of $\fq[Y_1,\ldots, Y_s]$ and the affine $\fq$--variety
$W_r^{\bfs{a}}\subset \A^s$ that they define is isomorphic to the
affine space $\A^{n-r}$. We conclude that $W_r^{\bfs{a}}$ is a
nonsingular variety and $(\partial \bfs{S}^{\bfs{a}}/\partial
\bfs{Y})(\bfs{y})$ has full rank for every $\bfs{y}\in\A^s$, that
is, $S_{n-s}^{\bfs{a}}\klk S_{r-1}^{\bfs{a}}$ satisfy
$(\mathsf{H}_2)$.

Finally, we show that the polynomials $S_{n-s}^{\bfs{a}}\klk
S_{r-1}^{\bfs{a}}$ satisfy $(\mathsf{H}_3)$. Lemma \ref{lemma:
formula Hij} and Proposition \ref{prop: formula for Rj} imply that
the homogeneous component of highest degree of each
$R_{j}^{\boldsymbol{a}}$ is $a_nH_{j,n}$ for $n-s\leq j \leq r-1$.
Lemma \ref{lemma: rel between R^di and S^wt} shows that
$a_nH_{j,n}=S_j^{\bfs{a},\mathrm{wt}}(\Pi_1\klk \Pi_s)$, where
$S_j^{\bfs{a},\mathrm{wt}}$ is the component of highest weight of
$S_j^{\bfs{a}}$. Since $H_{j,n}$ is a monic element of
$\fq[\Pi_1,\ldots,\Pi_{n-j-1}][\Pi_{n-j}]$ of degree $1$ in
$\Pi_{n-j}$, it follows that $S_j^{\bfs{a},\mathrm{wt}}$ is an
element of $\fq[Y_1,\ldots,Y_{n-j-1}][Y_{n-j}]$ of degree $1$ in
$Y_{n-j}$. As a consequence,
$$\cfq[Y_1\klk Y_s]/(S_{n-s}^{\bfs{a}, \mathrm{wt}}\klk
S_j^{\bfs{a},\mathrm{wt}})\simeq\cfq[Y_1\klk Y_{n-j-1}]$$
for $n-s\leq j \leq r-1 $. Arguing as above we conclude that
$S_{n-s}^{\bfs{a}, \mathrm{wt}}\klk S_{r-1}^{\bfs{a},\mathrm{wt}}$
satisfy $(\mathsf{H}_1)$ and $(\mathsf{H}_2)$, namely
$S_{n-s}^{\bfs{a}}\klk S_{r-1}^{\bfs{a}}$ satisfy $(\mathsf{H}_3)$.

%Let be given integers $n$ and $s$ with $n<q$, $1\le s\le n-2$ and
%$2(s+1)\le n$. Let also be given $\bfs{a}:=(a_{n-1}\klk a_{n-s})$
%and set the polynomial $f_{\bfs{a}}:=T^n+a_{n-1}T^{n-1}+\cdots+
%a_{n-s}T^{n-s}\in\fq[T]$. Recall that our main objective is to
%determine the asymptotic behavior of the average value set
%$\mathcal{V}(n,s,\bfs{a})$ of (\ref{eq: average value set}).

Let $V_r^{\bfs{a}}\subset\A^r$ be the affine variety defined by
$R_{n-s}^{\bfs{a}}\klk R_{r-1}^{\bfs{a}}\in\fq[X_1,\ldots X_r]$.
Since $r-n+s\leq s \leq n-s-2$ and the polynomials
$S^{\bfs{a}}_{n-s},\ldots,S^{\bfs{a}}_{r-1}$ satisfy
$(\mathsf{H}_1)$, $(\mathsf{H}_2)$ and $(\mathsf{H}_3)$, we can
apply Corollary \ref{coro: nb point V_r distinct coordinates} in
this situation.
%on the number $|V_r^{\bfs{a}}(\fq)|$ of $\fq$--rational points of
%$V_r^{\bfs{a}}$:
%%
%\begin{equation}\label{eq: estimate Vr todos}
%    \big||V_r^{\bfs{a}}(\fq)|-q^{n-s}\big|\leq  14 D_r^3
%\delta_r^2(q+1)q^{n-s-2},
%  \end{equation}
%
More precisely, let $V_{r,=}^{\bfs{a}}$ be the set of points of
$V_r^{\bfs{a}}$ with at least two distinct coordinates taking the
same value, namely
$$V_{r,=}^{\bfs{a}}:=\bigcup_{1\le i<j\le r}
V_r^{\bfs{a}}\cap\{X_i=X_j\},$$ and set
$V_{r,\not=}^{\bfs{a}}:=V_r^{\bfs{a}}\setminus V_{r,=}^{\bfs{a}}$.
By Corollary \ref{coro: nb point V_r distinct coordinates} we deduce
that
\begin{equation}\label{eq: est coord distintas}
\big||V_{r,\not=}^{\bfs{a}}(\fq)|-q^{n-s}\big|\le 14 D_r^3
\delta_r^2(q+1)q^{n-s-2}+
\binom{r}{2}\,\delta_r\,q^{n-s-1},\end{equation}
where $D_r:= \sum_{j=n-r+1}^s(j-1)$ and
$\delta_r=\prod_{j=n-r+1}^sj=s!/(n-r)!$. From \eqref{eq: est coord
distintas} we obtain the following estimate for $\chi(\bfs{a},r)$,
which constitutes the essential step in order to determine the
asymptotic behavior of $\mathcal{V}(n,s,\bfs{a})$.
\begin{theorem}\label{theorem: estimate chi(a,r)}
Let $n,r,s$ be integers with $1\leq s \leq n-2$ and $2(s+1)\leq n$.
For $n-s+1\le r\le n$ we have
$$
\left|\chi(\bfs{a},r)-\frac{q^{n-s}}{r!}\right|\le
\frac{r(r-1)}{2r!}\,\delta_rq^{n-s-1} +\frac{14}{r!} D_r^3
\delta_r^2(q+1)q^{n-s-2},$$
where $D_r:= \sum_{j=n-r+1}^s(j-1)$ and
$\delta_r=\prod_{j=n-r+1}^sj=s!/(n-r)!$.
\end{theorem}

Finally, combining Proposition \ref{prop: reduction value sets to
interp sets} and Theorem \ref{theorem: estimate chi(a,r)} we obtain
the following result.
\begin{corollary}\label{coro: estimate V(d,s,a) without analysis}
With assumptions as in Theorem \ref{theorem: estimate chi(a,r)}, we
have
$$\bigg|\mathcal{V}(n,s,\bfs{a})-\mu_n\,q-\frac{1}{2e}\bigg|\le
\frac{1}{2(n-s-1)!}+\frac{7}{q}+\sum_{r=n-s+1}^{n}
\Bigg(\frac{r(r-1)}{2r!}\,\delta_r+\frac{14}{r!} D_r^3
\delta_r^2(1+q^{-1})\Bigg).$$
\end{corollary}
\begin{proof}
From Proposition \ref{prop: reduction value sets to interp sets} we
deduce that
$$
\mathcal{V}(n,s,\bfs{a})-\mu_n\,q=\sum_{r=1}^{n-s}(-1)^{r-1}\bigg(\binom
{q} {r} q^{1-r}-\frac{q}{r!}\bigg)+\sum_{r=n-s+1}^{n} (-1)^{r-1}
\Bigg(\frac{\chi(\bfs{a},r)}{q^{n-s-1}}-\frac{q}{r!}\Bigg).
$$
By elementary calculations it can be seen that (see \cite[Corollary
4.2]{CeMaPePr14} for details)
$$\Bigg|\sum_{r=1}^{n-s}(-1)^{r-1}\bigg(\binom {q} {r}
q^{1-r}-\frac{q}{r!}\bigg)-\frac{1}{2e}\Bigg|\le
\frac{1}{2(n-s-1)!}+\frac{7}{q}.$$
Therefore, the corollary readily follows from Theorem \ref{theorem:
estimate chi(a,r)}.
\end{proof}

Finally, an analysis of the sum in the right--hand side of the
estimate of Corollary \ref{coro: estimate V(d,s,a) without analysis}
yields the following result, whose proof can be seen in \cite[\S
4.2]{CeMaPePr14}.
\begin{theorem}\label{theorem: final main result}
With assumptions as in Theorem \ref{theorem: estimate chi(a,r)}, we
have
$$%\begin{equation}\label{eq: upper bound error term}
\left|\mathcal{V}(n,s,\bfs{a})-\mu_n\, q-\frac{1}{2e}\right|\le
\frac{(n-2)^5e^{2\sqrt{n}}}{2^{n-2}} +\frac{7}{q}. $$%\end{equation}
\end{theorem}
%For the proof of Theorem \ref{theorem: final main result} we refer
%the reader to \cite{CeMaPePr14}.

Concerning the behavior of the bound of Theorem \ref{theorem: final
main result}, let $f:\Z_{\ge 4}\to\R$,
$f(n):=e^{2\sqrt{n}}(n-2)^52^{-n}$. Then $f$ is a unimodal function
which reaches its maximum value at $n_0:=14$, namely $f(n_0)\approx
1.08\cdot 10^5$. Furthermore, it is easy to see that
$\lim_{n\to+\infty}f(n)=0$, and indeed for $n\ge 51$, we have $f(n)<
1$.

This result constitutes an improvement of (\ref{eq: average value
set}) in several aspects. The first one is that it holds without any
restriction on the characteristic $p$ of $\fq$, while (\ref{eq:
average value set}) holds for $p>n$. The second aspect is that we
show that $\mathcal{V}(n,s,\bfs{a})=\mu_n\, q+\mathcal{O}(1)$, while
(\ref{eq: average value set}) only asserts that
$\mathcal{V}(n,s,\bfs{a})=\mu_n\, q+\mathcal{O}(q^{1/2})$. Finally,
we obtain an explicit expression for the constant underlying the
$\mathcal{O}$--notation with a good behavior. On the other hand, it
must be said that our result holds for $s\le n/2-1$, while (\ref{eq:
average value set}) holds for $s$ varying in a larger range of
values.

\bibliographystyle{amsalpha}

%\bibliography{tesis}
\bibliography{refs1,finite_fields}
\end{document}